\documentclass{article}
\usepackage{amssymb, amsmath}
\usepackage{amsthm}
\newtheorem{theorem}{Theorem}
\usepackage{mathtools}

\newtheorem{lemma}{Lemma}
\newtheorem{proposition}{Proposition}
\newtheorem{conjecture}{Conjecture}
\newtheorem{corollary}{Corollary}
\usepackage{bbm}
\usepackage[margin=1in]{geometry}
\usepackage{tikz}

\newcommand{\PetersenVertices}{
  \foreach \i/\ang in {1/90,2/18,3/-54,4/-126,5/162}{
    \node (v\i) at (\ang:2.35) {};
  }
  \foreach \i/\ang in {6/54,7/-18,8/-90,9/-162,10/126}{
    \node (v\i) at (\ang:1.25) {};
  }
}

\title{Nordhaus-Gaddum upper bounds for graph connectivity parameters}

\author{Mark Kempton\footnote{Brigham Young University, Provo, UT, \emph{mkempton@mathematics.byu.edu}}, 
Sibi Muthuprakash\footnote{Brigham Young University, Provo, UT, \emph{jmcbp@mathematics.byu.edu}}, 
and Xavier Zaitzeff\footnote{Brigham Young University, Provo, UT, \emph{xz224@byu.edu}}}

\date{}

\begin{document}

\maketitle

\begin{abstract}
    We examine upper bounds on Nordhaus-Gaddum type problems for parameters related to graph connectivity.  Our main result is that for a graph $G$ on $n$ vertices where both $G$ and its complement $G^c$ are connected, then the sum of the  algebraic connectivity of $G$ and the algebraic connectivity of $G^c$ cannot exceed $n-3$ (with finitely many exceptions with a small number of vertices).  We obtain similar results for the isoperimetric number of a graph, and explore similar Nordhaus-Gaddum type questions for the Cheeger constant and the second eigenvalue of the normalized Laplacian matrix. 
\end{abstract}



\section{Introduction}

The combinatorial Laplacian matrix of a graph $G$ is defined as $L(G) = D(G) - A(G)$ where $D(G)$ is the diagonal degree matrix and $A(G)$ is the adjacency matrix of G. As the Laplacian matrix is positive semi-definite, the eigenvalues of the Laplacian matrix are typically ordered $0 = \mu_1(G) \leq \mu_2(G) \dots \leq \mu_n(G)$. 
The Laplacian has several nice properties that make it useful for applications such as community detection, measuring connectivity, image segmentation, modeling diffusion in a graph, and other areas. Its smallest eigenvalue $\mu_1(G)=0$ with the all-ones vector the associated eigenvector. In fact, the multiplicity of the eigenvalue 0 is the number of connected components in the graph. Thus, $\mu_2(G)$, the second smallest eigenvalue of the Laplacian, is strictly greater than 0 if and only if G is connected. As such, $\mu_2(G)$ is called the \emph{algebraic connectivity} of $G$ and can be thought of as a measure of how well-connected the graph is (see, for instance, the seminal work of Fiedler in \cite{fiedler}). Generally, a small $\mu_2(G)$ implies the graph is poorly connected and a large $\mu_2(G)$ that it is well connected. These properties and more have inspired significant research into the eigenvalues of the Laplacian and the algebraic connectivity specifically.

We will analyze the algebraic connectivity from the perspective of a Nordhaus-Gaddum problem. A Nordhaus-Gaddum type problem for a given graph invariant involves finding an upper and/or lower bound on a quantity (such as a sum or a product) involving the graph invariant of the graph and its complement simultaneously. For a  simple undirected graph $G$, the complement, $G^c$, is defined as the graph with the same vertices as the original and whose edges are exactly the edges not in the original graph. By the properties of the Laplacian, if $G$ and $G^c$ are both connected then both $\mu_2(G)$ and $\mu_2(G^c)$ are nonzero.
The study of Nordhaus-Gaddum type questions traces back to the work of Nordhaus and Gaddum in 1956 \cite{nordhaus1956complementary} where they showed that for the chromatic number of a graph, $\chi(G)$,  
\[2\sqrt{n} \leq \chi(G) +\chi(G^c) \leq n+1\]
\[n \leq \chi(G)\chi(G^c) \leq \left(\frac{n+1}{2}\right)^2.\]
Since then, several Nordhaus-Gaddum inequalities for many different kinds of graph invariants and graph parameters have been studied. See \cite{aouchiche2013survey} for a survey of such Nordhaus-Gaddum type questions.  See in particular \cite{barrett2022new,faught2024nordhaus,kim2026bounds} for studies of Nordhaus-Gaddum questions (mostly lower bounds) for parameters related to graph connectivity.

 
 When we know a Nordhaus-Gaddum type bound, the associated extremal graphs tell us what graph structure optimizes the given invariant for the graph and its complement at the same time, which can have value in real world graph applications. For example, consider the hypothetical scenario of two communications networks which must remain disjoint. These networks want to optimize the speed to get from one point to another, but cannot share edges. 
 The optimal structure of the respective systems to this problem would maximize connectivity of graphs that are complementary. The Nordhaus-Gaddum formulation offers intuition on how these parallel structures interact, and how improving one network might necessarily be detrimental to the other. An upper bound on a Nordhaus-Gaddum problem related to graph connectivity informs how strong these networks can be simultaneously. 

Closely related to the Nordhaus-Gaddum problem for algebraic connectivity is the problem of the Laplacian spread. The Laplacian spread is defined as $s(G) = \mu_n(G) - \mu_2(G)$. The Laplacian spread conjecture was posited in 2011 by Zhai et al.~\cite{zhai2011laplacian} and then proved in 2021 by Einollahzadeh and Karkhaneei \cite{einollahzadeh2021lower}. It states that for any simple undirected graph of order $n \geq 2$,
\begin{equation}
    \mu_n(G)-\mu_2(G) \leq n-1 \label{eq:1}
\end{equation}
with equality if and only if $G$ or $G^c$ is isomorphic to the join of an isolated vertex with a disconnected graph. Einollahzadeh and Karkhaneei's approach involved relating the eigenvalues of a graph with its complement by the well-known relationship
\begin{equation}
    \mu_i(G) = n - \mu_{n-i+2}(G^c). \label{eq:2}
\end{equation} Using this, we can rewrite the Laplacian spread conjecture in the form of a Nordhaus-Gaddum type problem.
\begin{equation}
    \mu_2(G)+\mu_2(G^c) \geq 1 \label{eq:3}
\end{equation} 
This gives a lower bound for the Nordhaus-Gaddum problem we are interested in. The primary purpose of this paper is to address upper bounds. 

While the algebraic connectivity of a graph is our primary concern in this paper, we will also look at other parameters related to connectivity. One such graph invariant is the isoperimetric number. Considering the communication networks from before, the networks should be resistant to failures. Knowing how many edges are required to fail before the graph becomes disconnected is another measure of connectivity. The isoperimetric number of $G$ quantifies this and is defined as 
\[
i(G) = \min_{X \subset V}{\frac{|\partial(X)|}{\min\{|{X}|,|\overline{X}|\}}}
\] where $\partial(X)$ denotes the ``boundary" of $X$; i.e.~the set of edges with one endpoint in $X$ and the other not in $X$. 
Much like $\mu_2(G)$, if $G$ is disconnected then $i(G)=0$. This definition can be thought of as the sparsest cut of a graph, telling us the minimum number of edges required to cut the graph into two disconnected components. A high $i(G)$ implies many edges need to be removed and a low $i(G)$ indicates a small number of edges to cut. Mohar \cite{MOHAR1989274} proved the following connection between the algebraic connectivity of a graph and the isoperimetric number: 
\begin{equation}
    \frac{i(G)^2}{2\Delta}\leq \mu_2(G) \leq 2i(G). \label{eq:3}
\end{equation}
From this, we can analyze Nordhaus-Gaddum inequalities on the isoperimetric number and apply them to the algebraic connectivity. 

Two other parameters we will consider are the second eigenvalue of the normalized Laplacian matrix, and the Cheeger constant. The normalized Laplacian is defined as $\mathcal{L} = D^{-1/2}LD^{-1/2}$. We denote its second smallest eigenvalue as $\lambda_2(G)$ much like we did with the Laplacian. The normalized Laplacian also has the property of $\lambda_1(G) = 0$ and the multiplicity of 0 is the number of components of $G$. This definition normalizes the weights of the edges based on the degree of the corresponding vertices. Similarly, the Cheeger constant $h(G)$, a normalized version of the isoperimetric number, provides a like bound on $\lambda_2(G)$ associated with the normalized Laplacian \cite{Fanchungbook}: 
\begin{equation}\label{eq:cheegerineq}
\frac{h(G)^2}{4} \leq \lambda_2 \leq 2h(G). 
\end{equation}
We note that work in \cite{faught2024nordhaus} gives lower bounds on Nordhaus-Gaddum questions relating to $\lambda_2$, $h(G)$, and $i(G)$.



The purpose of this paper is to prove upper bounds on the noted graph invariants in the context of Nordhaus-Gaddum formulations and identify the extremal graph families. The paper will be structured as follows. First we will show for the isoperimetric number $i(G) + i(G^c) \leq \frac{n}{2}$ in general, and $i(G)+i(G^c) \leq \frac{n}{2} - \frac{2}{n}$ when the graph and its complement are both connected. We will then use these results and other techniques to analyze the algebraic connectivity of $G$ and $G^c$ to find upper bounds. We observe $\mu_2(G) + \mu_2(G) \leq n$ in general, and then prove that, with finitely many exceptions, $\mu_2(G) + \mu_2(G) \leq n-3$ given both $G$ and $G^c$ are connected. We will identify the extremal families of graphs for $\mu_2(G)$ and $\mu_2(G^c)$ and prove how $k$-regularity gives strong connectivity for the Nordhaus-Gaddum problem. Following we conjecture some possible Nordhaus-Gaddum upper bounds for the Cheeger constant and $\lambda_2(G)$. The extremal graphs will be identified and their structures addressed.




\section{Isoperimetric Number}
    
    We begin analyzing the isoperimetric number by gathering some experimental data.  In Figure \ref{fig:isoplots}, we plot the points $(i(G), i(G^c))$ for all graphs (up to isomorphism) on $n$ vertices for $n=4, 5, 6, 7, 8, 9$. Since $i(G)=0$ means $G$ is disconnected, points along the $x$-axis or $y$-axis mean that $G$ or $G^c$ is disconnected. In the plots, we observe a clear grouping of graphs along the axes apart from the interior of the plots. We first give a general upper bound on $i(G)$. Throughout the paper, we will denote the complete graph on $n$ vertices as $K_n$, and the empty graph (that is, the graph with only isolated vertices and no edges) on $n$ vertices as $E_n$. 
    \begin{figure}[h]
    \includegraphics[width=0.3\linewidth]{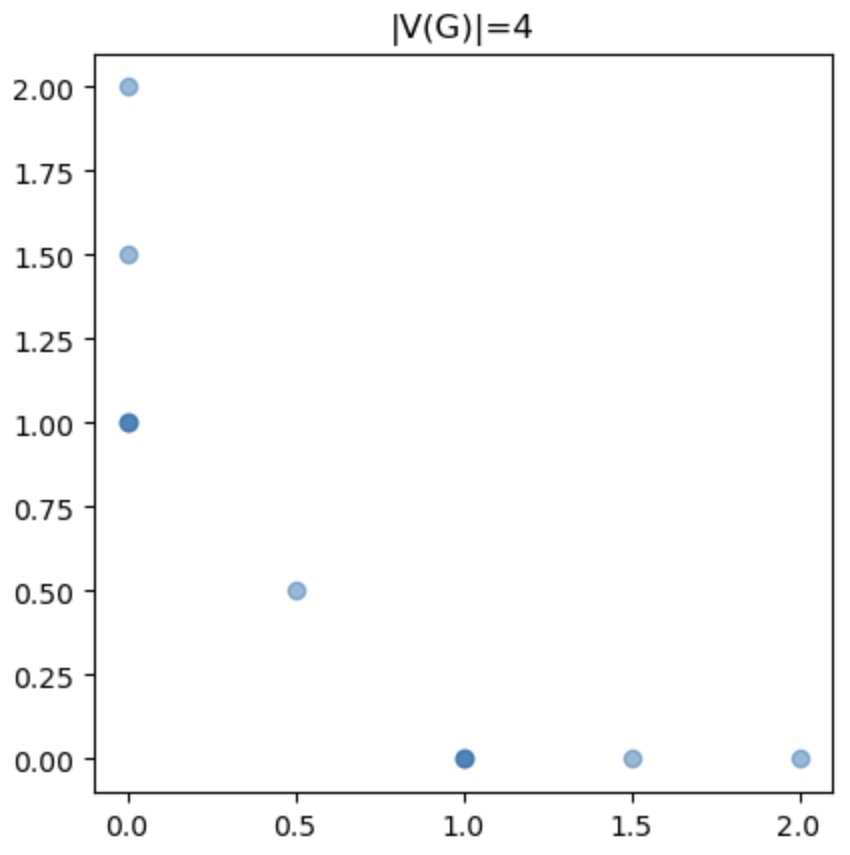}
    \centering
    \includegraphics[width=0.3\linewidth]{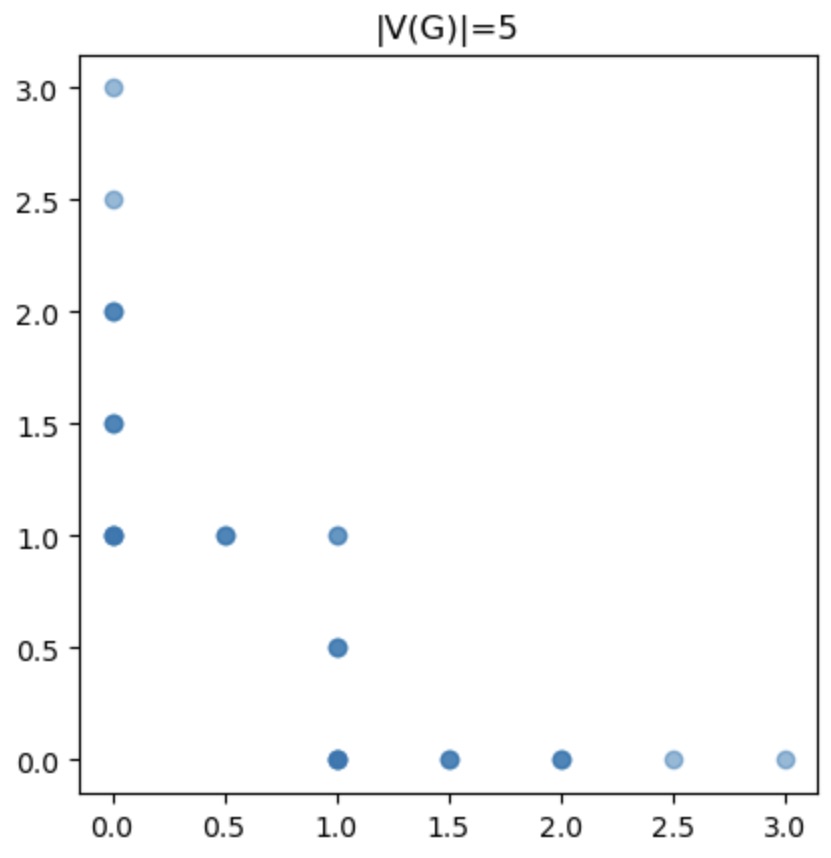}
    \includegraphics[width=0.3\linewidth]{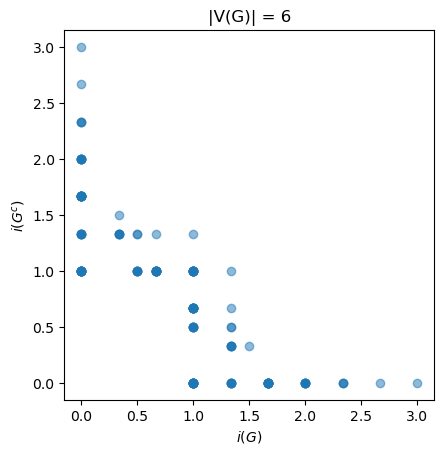}

    \includegraphics[width=0.3\linewidth]{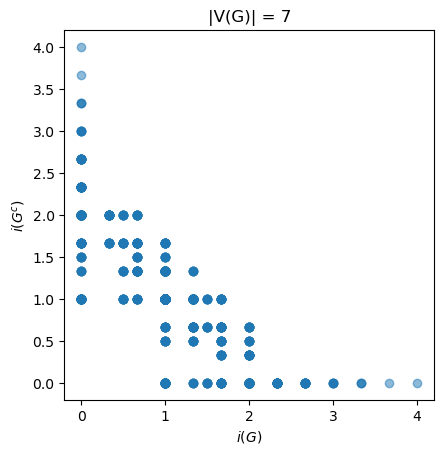}
    \centering
    \includegraphics[width=0.3\linewidth]{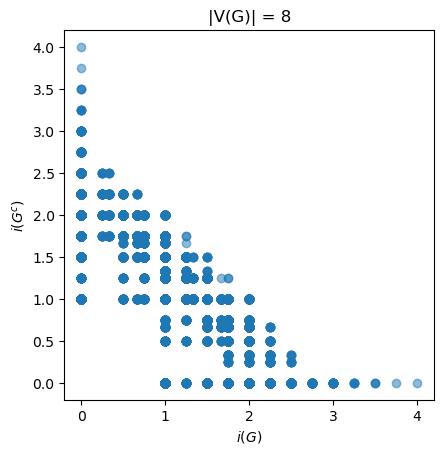}
    \includegraphics[width=0.3\linewidth]{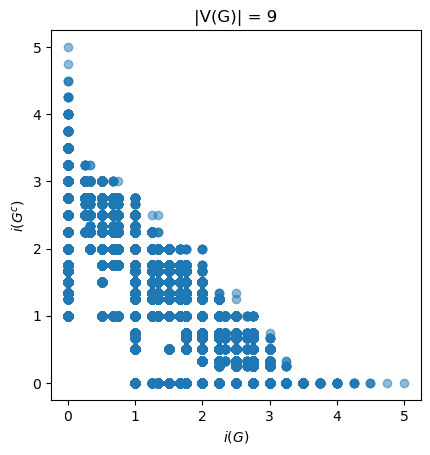}
    \caption{Plots of the isoperimetric number for all graphs up to isomorphism on 4-9 vertices. Note $i(G)+i(G^c) \geq 1$ unless $G=P_4$ (see \cite{faught2024nordhaus}).}\label{fig:isoplots}
    \end{figure}

    \begin{lemma}
        $i(G) \leq\lceil\frac{n}{2}\rceil$ with equality if and only if $G=K_n$. \label{iso:1}
    \end{lemma}

    \begin{proof}
        $(\implies)$Assume $G\neq K_n$. Let $X \subset V $ and $|X| \leq \lceil\frac{n}{2}\rceil$. Then the edge set is bounded above by a value that is less than the product of the cardinalities of $X$ and $\bar X$. That is, $\partial (X,\bar X)<|X||\bar X|-\alpha$, where $\alpha>0$. We choose a set $X$ that has cardinality $\frac{n}{2}$. Then we see that $i(G)=\min\{\frac{\partial(X,\bar X)}{\min{|X||\bar X|}}\}$=$\frac{\frac{n}{2}.\frac{n}{2}-\alpha}{\frac{n}{2}}$=$\frac{n}{2}-\alpha<\frac{n}{2}$. \\
        $(\impliedby)$ Let $G=K_n$. The optimal cut will be when the cardinality is exactly at half of the degree of the graph. So we have that $\partial(X,\bar X)=|X||\bar X|=\frac{n}{2}.\frac{n}{2}=\frac{n^2}{4}$. Dividing both sides by $|X|=\frac{n}{2}$ yields $i(G)=\frac{n}{2}$.
    \end{proof}

    Lemma \ref{iso:1} tells us an absolute bound on isoperimetric number of a graph. We see the graph most resistant to edge cuts is the complete graph $K_n$. Now we connect the isoperimetric number of a graph with the isoperimetric number of its complement.

    \begin{lemma}
        If $i(G) \geq \alpha$ for some $\alpha \geq 0$, then $i(G^c)\leq \lceil\frac{n}{2}\rceil-\alpha$. \label{iso:2}
    \end{lemma}

    \begin{proof}
        Without loss of generality, assume $|X|=\frac{n}{2}$ since we want to minimize the isoperimetric ratio. Now, we set $i(G) = \alpha$ for some $\alpha \geq 0$. Then $i(G)=\frac{\partial(X,\bar X)}{|X|} \geq \alpha$. Then we see that $\partial(X,\bar X)\geq\alpha |\bar{X}| \geq \frac{n\alpha}{2}$. We denote the edge set on the $G^c$ as $\partial_{G^c}(X, \bar X)$. Since $|X|=\frac{n}{2}, |\bar X|=\frac{n}{2}$, the maximum number of edges possible between $X$ and its complement will be $\frac{n^2}{4}$. Then we see that $|\partial_{G}(X,\bar X)| < \frac{n^2}{4}-\frac{n\alpha}{2}$, as $|\partial_{G^c}(X,\bar X)|>\frac{n\alpha}{2}$. Since $|\partial_{G^c}(X,\bar X)|\leq\frac{n^2}{4}-\frac{n\alpha}{2}=\frac{n}{2}(\frac{n}{2}-\alpha)=|X|(\frac{n}{2}-\alpha)$, dividing both sides by $|X|$ gives $\frac{|\partial_{G^c}(X,\bar X)|}{|X|}\leq\frac{n}{2}-\alpha$, from which we see that i($G^c$)$<$$\frac{n}{2}-\alpha$.  
        In the case when $n$ is odd, apply the same reasoning with $|X|=\frac{n-1}{2}$ and for $i(G)> \alpha$, $i(G^c)<\frac{n+1}{2}-\alpha$.
    \end{proof}

Lemma \ref{iso:2} provides intuition in how this invariant pair interacts. As an edge is lost in a graph, the complement gains an edge. The same cut that would disconnect a graph would have to cut all the remaining edges possible between the two vertex subsets found in the complement. We can see this visually in Figure \ref{fig:isoplots}. Now combining Lemma \ref{iso:1} with Lemma \ref{iso:2}, we prove a Nordhaus-Gaddum upper bound on $i(G)$ as a main result for this section.

    \begin{theorem}\label{thm:iso}
        $i(G)+i(G^c)\leq\lceil\frac{n}{2}\rceil$ with equality if and only if $G = K_n$ or $G= E_n$.
    \end{theorem}

    \begin{proof}
        We first set $i(G)=\alpha$. Then by Lemma \ref{iso:2}, we see that $i(G^c)\leq\frac{n}{2}-\alpha$. Combining the two conditions, we get $i(G)+i(G^c)<\alpha+\frac{n}{2}-\alpha=\frac{n}{2}$. This changes slightly when $n$ is odd, $i(G)+i(G^c)<\frac{n+1}{2}$. By Lemma \ref{iso:1}, the equality holds when G is $K_n$.
        $E_n$ is the complement of $K_n$ and so apply the same argument but swapping $G$ and $G^c$.
    \end{proof}

    This Nordhaus-Gaddum upper bound is universal for all graphs. The plots in Figure \ref{fig:isoplots} suggest that a stronger bound holds if both $G$ and $G^c$ are connected. 
    
    \begin{theorem}\label{thm:iso_connected}
        If $G$ and $G^c$ are connected, then $i(G)+i(G^c)<\lceil\frac{n}{2}\rceil-\frac{2}{n}$.
    \end{theorem}

    \begin{proof}
        Let \(i(G) = \alpha\). Then 
        $i(G) = \frac{|\partial(X, \bar X)|}{|X|} = \alpha$, then $|\partial(X,\bar X)|=\frac{n\alpha}{2}$.
        In the complement, since $G$ and \(G^c\) must be connected the boundary set must be at least one less edge than the complete boundary set. So
        $\partial(X, \bar X)$ $\leq$ $\frac{n^2}{4}-1-\frac{n\alpha}{2}$
        Thus, 
         $i(G^c) = \frac{\partial(X, \bar X)|}{|X|}\leq\frac{\frac{n^2}{4}-1-\frac{n\alpha }{2}}{\frac{n}{2}}=\frac{n}{2}-\frac{2}{n}-\alpha$
         So $i(G) + i(G^c) \leq \frac{n}{2} - \frac{2}{n}$.
    \end{proof}

    The data displayed in Figure \ref{fig:isoplots} does seem to indicate that it may be possible to improve Theorem \ref{thm:iso_connected} slightly.  We leave it as an open question if there is a tighter bound when $G$ and $G^c$ are connected.

\section{Algebraic Connectivity}

\begin{figure}[h]
    \includegraphics[width=0.3\linewidth]{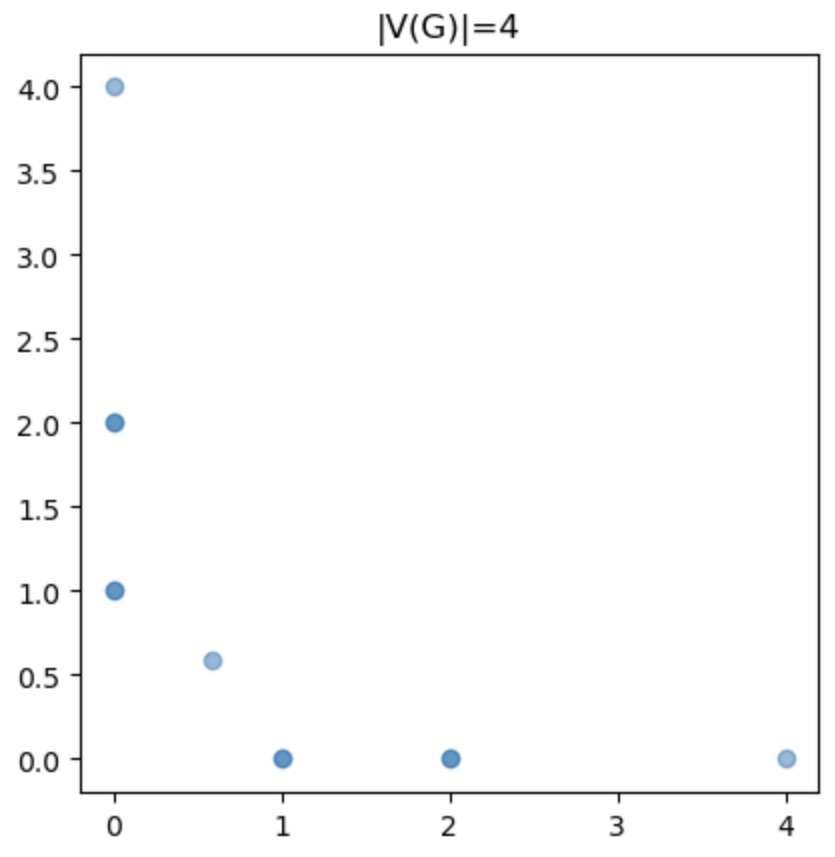}
        \centering
    \includegraphics[width=0.3\linewidth]{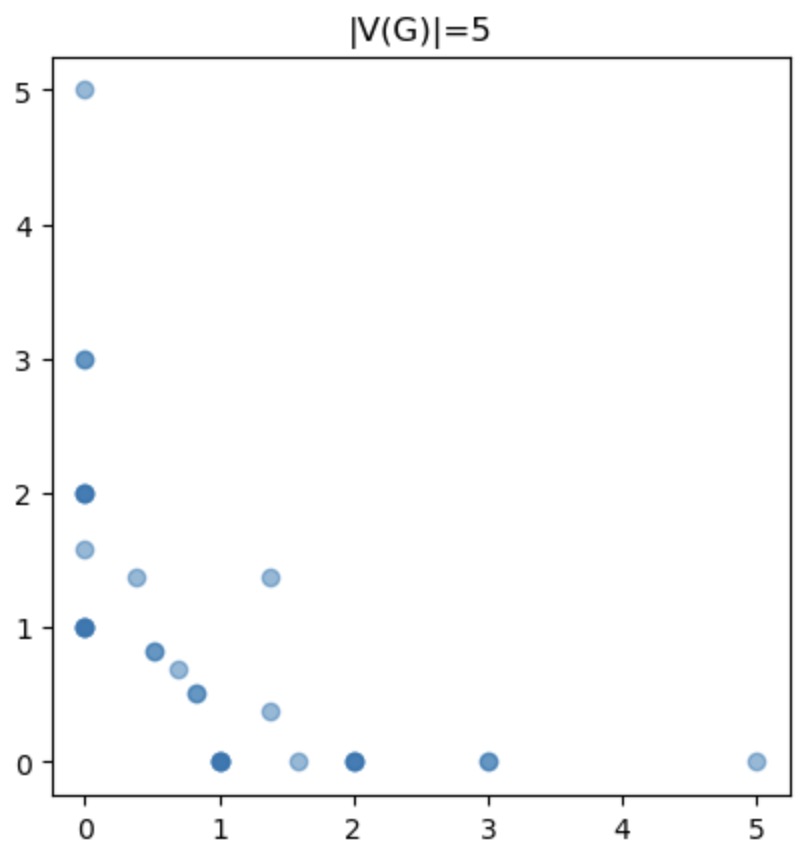}
    \includegraphics[width=0.3\linewidth]{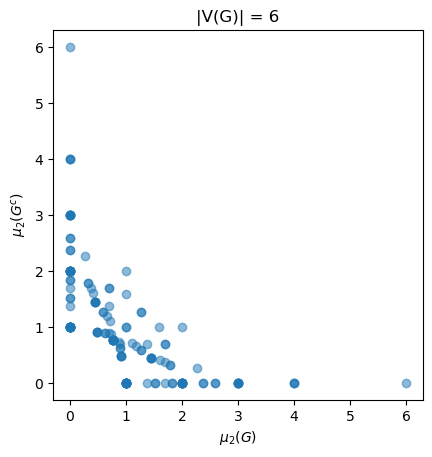}
    
    \includegraphics[width=0.3\linewidth]{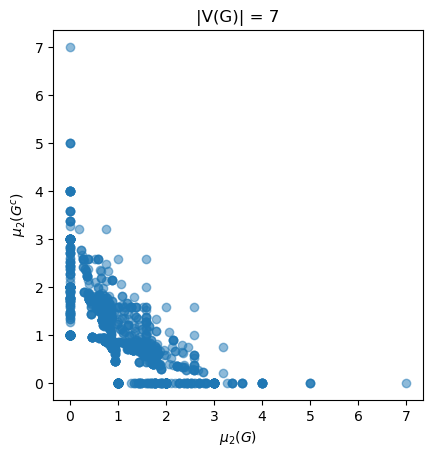}
    \centering
    \includegraphics[width=0.3\linewidth]{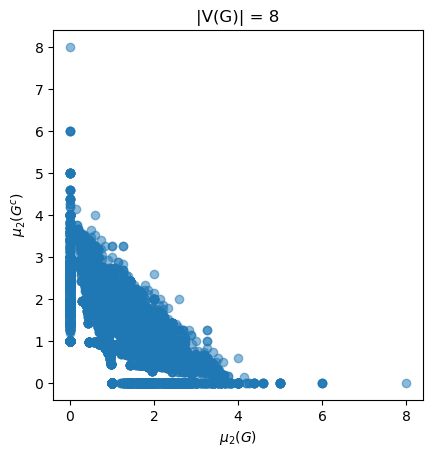}
    \includegraphics[width=0.3\linewidth]{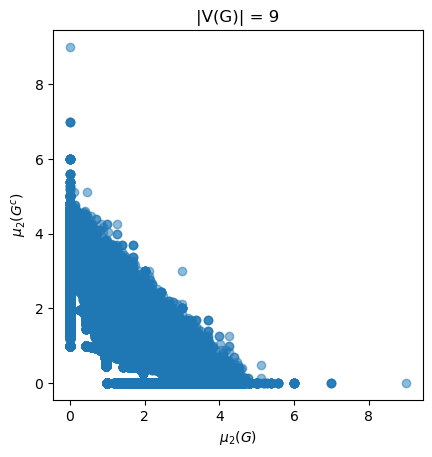}
    \caption{Plots of algebraic connectivity up to isomorphism for $n=4,5,6,7,8,9$}\label{fig:mu_plots}
\end{figure}

We now turn our attention to bounding the sum of the algebraic connectivity of a graph and its complement. Like previously, we plotted the $(\mu_2(G), \mu_2(G^c))$ over the same choices of $n$. By \eqref{eq:3}, $\mu_2(G)$ is bounded above and below by $i(G)$ so we expect parallels between optimal graph structures seen for $i(G)$ and those for $\mu_2(G)$.  From Theorem \ref{thm:iso} and equation (\ref{eq:3}), it is immediate that the sum of the algebraic connectivity of $G$ and $G^c$ will be upper bounded by around $n$, but this can be proven using much more elementary means which we do below. 

Throughout this section, we will use $\Delta$ to refer to the maximum degree of the graph and $\delta$ as the minimum degree. It is a well established that $\mu_n \geq\Delta+1$ and $\mu_2 \leq \delta$ for $G \neq K_n$.

\begin{lemma}[see, for instance, \cite{brouwer2011spectra} Propositions 1.7.2 and 3.9.3]\label{lem:degree}
     If $G\neq E_n$, $\mu_n\geq\Delta+1$, with strict inequality if $G$ is connected and $\Delta\neq n-1$. If $G\neq K_n$, then $\mu_2(G) \leq \delta(G)$. 
\end{lemma}

\begin{lemma}[see \cite{CvetkovicDoobSachs1980}]\label{lem:multn-1} The only simple graph on $n$ vertices with a nonzero eigenvalue of multiplicity $n-1$ is the complete graph. 
\end{lemma}

\begin{proposition}
Let $G$ be any graph with $n$ vertices.  Then $\mu_2(G)+\mu_2(G^c)\leq n$ with equality if and only if $G=K_n$ or $G=E_n$. 
\label{mu:1}
\end{proposition}

\begin{proof}

 Note that by (\ref{eq:2}), $\mu_2(G)+\mu_2(G^c)=\mu_2(G)+n-\mu_n(G) = n-(\mu_n(G)-\mu_2(G))\leq n$ as $\mu_n(G)-\mu_2(G)\geq0$.  We get equality if and only if $\mu_n(G)=\mu_2(G)$ which means there is an eigenvalue of multiplicity $n-1$.  By Lemma \ref{lem:multn-1}, this occurs if and only if $G$ is $K_n$ (or $E_n$).
\end{proof}


 

The above theorem demonstrates the largest sum as well as the associated graphs that achieve it. We can still see from the plots there exist tighter bounds given some restrictions. First, there is a large gap between what is achieved by $K_n$ and the next largest sum. We now consider the next largest sum.


\begin{proposition}
If $G\neq K_n$ and if $G^c$ is disconnected, then $\mu_2(G)+\mu_2(G^c) \leq n-2$ with equality if and only if $G^c$ is a union of disjoint edges and disjoint vertices.  Moreover, if $G^c$ is disconnected and not a union of disjoint edges and vertices, then $\mu_2(G)+\mu_2(G^c)\leq n-3$.
\end{proposition}

\begin{proof}
Since $G^c$ is disconnected, $\mu_2(G^c)=0$. However, since $G$ is not complete, the highest possible value for the minimum degree $\delta$ is $n-2$, so by Lemma \ref{lem:degree}, $\mu_2(G)\leq n-2$. Thus $\mu_2(G)+\mu_2(G^c) \leq n-2$.  We can see that if equality holds, then the minimum degree of $G$ must be $n-2$, and so the maximum degree of $G^c$ is 1, and thus the graph must be as described.  Furthermore, if $G$ and $G^c$ are as claimed, we can directly compute that $\mu_2(G)=n-2$ using (\ref{eq:2}). Finally, if $G^c$ has any vertex of degree 2 or more, then the minimum degree of $G$ is at most $n-3$ and the claim follows.
\end{proof}




As with the isoperimetric number, the data in Figure \ref{fig:lapl_plots} seems to indicate that a stronger bound holds if both $G$ and $G^c$ are connected.  From Theorem \ref{thm:iso_connected} and equation (\ref{eq:3}), we immediately get an upper bound of around $n-\frac4n$.  However, we can do much better, as we will show below.

To prove our main result, we require a general bound for the Laplacian spread and a bound for $k$-regular graphs (graphs for which each vertex has degree $k$, so $\Delta = \delta$).  For our general bound, Andrade et al. in 2015 proved a bound on the Laplacian spread \cite{ANDRADE2015494}. 

\begin{lemma}[Theorem 4.2 of \cite{ANDRADE2015494}]\label{lem:spread_bound}
Let $G$ be a graph of order $n\geq 3$ with at least one edge, such that $\delta = \delta(G)$ and $\Delta=\Delta(G)$. Then
\begin{equation}
    \mu_n-\mu_2 \geq \sqrt{(\Delta - \delta)^2 + 2(\Delta+\delta) - \left(2\frac{\Delta+1}{n}\right)^2\left(\frac{\delta-1}{\Delta+1}n + 1\right)}.
    \label{eq:5}
\end{equation}

In particular, if G is $k$-regular, then 
\[\mu_n-\mu_2 \geq \sqrt{4k - \left(2\frac{k+1}{n}\right)^2 \left(\frac{k-1}{k+1}n + 1\right)}.\]
\end{lemma}

We can reformulate this into the Nordhaus-Gaddum problem.
\[
\mu_2(G) + \mu_2(G^c) \leq n- \sqrt{(\Delta - \delta)^2 + 2(\Delta+\delta) - \left(2\frac{\Delta+1}{n}\right)^2\left(\frac{\delta-1}{\Delta+1}n + 1\right).}
\]

\[
\mu_2(G) + \mu_2(G^c) \leq n- \sqrt{4k - \left(2\frac{k+1}{n}\right)^2 \left(\frac{k-1}{k+1}n + 1\right).}
\]

 Likewise, we use a Laplacian spread bound on $k$-regular graphs found in \cite{GOLDBERG200668}. 
 The complement of a regular graph is also regular, specifically ($n-k-1$)-regular. So looking at $k$-regular graphs also is analyzing ($n-k-1$)-regular graphs in the complement.

\begin{lemma}[Theorem 3.6 of \cite{GOLDBERG200668}]\label{lem:spred_bound_regular}
    Let $G$ be a connected $k$-regular graph with n vertices. Then:
    \[
    \mu_n(G) - \mu_2(G) \geq 2\sqrt{\frac{k(n-k-1)}{n}}.
    \]
    \label{thr:reg_laplacian}
\end{lemma}

We reformulate this as a Nordhaus-Gaddum upper bound as 
\[
\mu_2(G)+\mu_2(G^c) \leq n - 2\sqrt{\frac{k(n-k-1)}{n}}.
\]



\begin{theorem}\label{thm:alg_conn} Let $G$ have $n$ vertices and assume that $G$ and $G^c$ are both connected. Then for $n\geq8$, \[\mu_2(G) + \mu_2(G^c) \leq n-3.\]  Moreover, equality can only hold if $G$ is regular.
\end{theorem}

\begin{proof}

Case 1: $\Delta = \delta =k$ (that is, $G$ is $k$-regular).

We apply Lemma \ref{thr:reg_laplacian}. Given the Nordhaus-Gaddum upper bounds, we must show under what conditions $2\sqrt{\frac{k(n-k-1)}{n}} \geq 3$. For $G$ and $G^c$ to be connected, $2\leq k\leq n-3$. We first consider $k=2$ and $k=n-3$. The only 2-regular graph where $G$ and $G^c$ are connected is $C_n$ and $G^c$ is the only connected $n-3$-regular graph.  Their spectra is well known and $\mu_j(C_n) = 2(1 - \cos\frac{2\pi j}{n})$ for $j = 0,1,2,\dots,n-1$. So $\mu_2(C_n) =2-2\cos(\frac{2\pi}{n})$ and $\mu_2(C_n^c) = n -2+2\cos(\frac{2\pi\lfloor n/2\rfloor}{n})$. Then $\mu_2(C_n) + \mu(C_n^c) = n-2(\cos(\frac{2\pi}{n})-\cos(\frac{2\pi \lfloor n/2\rfloor}{n})) \leq n-3$ when $n\geq 6$.

We now look at $3\leq k \leq n-4$.
The bound $2\sqrt{\frac{k(n-k-1)}{n}}$ from Lemma \ref{thr:reg_laplacian}  its minimums (with respect to $k$) at the endpoints. Then evaluating the endpoints, $2\sqrt{\frac{k(n-k-1)}{n}} = 2\sqrt{\frac{3(n-4)}{n}} \geq 3$ when $n \geq 16$. For $n\leq 15$, we exhaustively checked all $k$-regular graphs where the values of $k$ did not meet this criteria. The only exceptional regular graph is $C_5$.

Case 2: $\Delta -\delta \geq 2$.

By Lemma \ref{lem:degree}$, \Delta+1 \leq \mu_n$ and since $G$ and $G^c$ are connected, then in fact $\Delta+1>\mu_n$. Also $\mu_2\leq\delta$ by the same lemma, so $(\mu_n-\mu_2) > \Delta+1 - \delta \geq 3$. Thus $\mu_2(G)+\mu_2(G^c)=\mu_2(G)+n-\mu_n(G) 
< n-3$.
Thus if $\Delta$ and $\delta$ differ by 2 or more then the result holds.

Case 3: $\Delta-\delta =1$.

We use the same formulation as the previous case. By Eq. \ref{eq:5},
\[
\mu_n-\mu_2 \geq \sqrt{(\Delta - \delta)^2 + 2(\Delta+\delta) - \left(2\frac{\Delta+1}{n}\right)^2\left(\frac{\delta-1}{\Delta+1}n + 1\right)}
\]
Substituting $\delta=\Delta-1$ yields
\[
\sqrt{(\Delta - \delta)^2 + 2(\Delta+\delta) - \left(2\frac{\Delta+1}{n}\right)^2\left(\frac{\delta-1}{\Delta+1}n + 1\right)} = \sqrt{4\Delta-1 - \left(2\frac{\Delta+1}{n}\right)^2\left(\frac{\Delta-2}{\Delta+1}n + 1\right)}.
\]

By direct computation, it can be shown if $\Delta = 2$ as $n\to\infty$, 
\[
\sqrt{4\Delta-1 - \left(2\frac{\Delta+1}{n}\right)^2\left(\frac{\Delta-2}{\Delta+1}n + 1\right)} = \sqrt{7-\frac{36}{n}} \to \sqrt{7} < 3
\]
However the only graphs with $\Delta=2$ and meet the connectedness condition are paths $P_n$. Their spectra is well known: we have 
$\mu_i(P_n) = 2(1-\cos(\frac{i\pi}{n}))$ for $i = \{0, 1,\dots, n-1\}$.
Thus

\[
\mu_2(P_n)+\mu_2(P_n^c) = \mu_2(P_n) + n - \mu_n(P_n)=n-2\cos\left(\frac{\pi}{n}\right)+2\cos\left(\frac{(n-1)\pi}{n}\right) = n-4\cos\left(\frac{\pi}{n}\right)
\]

For $n \geq 5$, $n-4\cos(\frac{\pi}{n}) \leq n-3$, so $P_4$ is the only path with connected complement that exceeds the $n-3$ bound.

For $\Delta=n-1$ the complements will be disconnected so we do not consider them here. If $\Delta=n-2$, then $\delta=n-3$. This means for all vertices with degree $\Delta$ will be connected to one vertex in $G^c$ and vertices with degree $\delta$ will be connected to two vertices. While maintaining our connectedness premise, the only family of graphs that meet this condition is $P_n$ which we have already considered. 

Now, in the remaining cases, namely $3\leq\Delta\leq n-3$, we need to show
\[
\sqrt{4\Delta - 1 - 4\left(\frac{\Delta+1}{n}\right)^2\left(\frac{\Delta-2}{\Delta+1}n + 1\right)} > 3
\] which, after some algebra, is equivalent to
\[
\Delta^{2}\left(-4n-4\right)+\Delta\left(4n^2+4n-8\right)+\left(-10n^2+8n-4\right) > 0.
\]
Since the left-hand side is concave as a function of $\Delta$, we need only show that the interval $[3,n-3]$ is contained in the part of the real line where this function is positive.  Using the quadratic formula, the function is 0 for
\[
\Delta =\frac{n^2+n-2 \mp n\sqrt{n^2-8n-5}}{2n+2}.
\]
Thus, we need to show
\[
\frac{n^2+n-2 + n\sqrt{n^2-8n-5}}{2n+2} > n-3
\]
and
\[
\frac{n^2+n-2 - n\sqrt{n^2-8n-5}}{2n+2} < 3.
\]

The first (after some algebra) is equivalent to
\[
n^3-11n^2-20n-8 > 0
\]
which can be seen to be true for $n\geq13$.
Thus the root will always be greater than $n-3$ for $n\geq13$.

Similarly, the second equation is equivalent to
\[
2n^3 - 14n^2-80n-64 > 0
\]
which can again be seen to be true for $n\geq13.$
Thus the root will always be less than $3$ for $n\geq 13$. Doing an exhaustive search over all graphs with $\Delta-\delta=1$ for $n=8, 9, 10,11,12$ yielded no graphs whose Nordhaus-Gaddum sum is greater than $n-3$.
\end{proof}
 
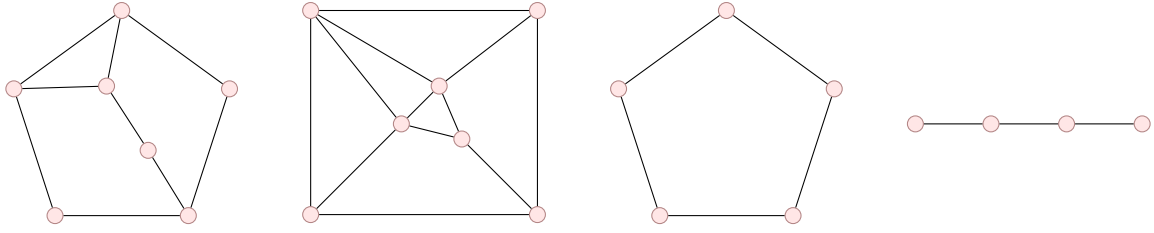
\begin{figure}
    \centering

\begin{tikzpicture}[
every node/.style={circle, fill=pink!40, draw=pink!70!black,
inner sep=0pt, minimum size=6pt},
every edge/.style={gray!70, thick}
]

\foreach \i in {1,...,5} {
\node (v\i) at ({90 + (\i-1)*72}:1.5cm) {};
}
\foreach \i/\j in {1/2, 2/3, 3/4, 4/5, 5/1} {
\draw (v\i) -- (v\j);
}

\node (u) at (-.2, .5) {};
\draw (u) -- (v1);
\draw (u) -- (v2);

\node (w) at (.35, -.35) {};
\draw (w) -- (v4);
\draw (w) -- (u);

\begin{scope}[xshift=4cm]

\node (v3) at (-1.5, 1.5) {};
\node (v4) at (1.5, 1.5) {};
\node (v5) at (1.5, -1.2) {};
\node (v6) at (-1.5, -1.2) {};

\node (v2) at (.2, .5) {};
\node (v1) at (-0.3, 0) {};
\node (v0) at (0.5, -0.2) {};

\draw (v3) -- (v4);
\draw (v4) -- (v5);
\draw (v5) -- (v6);
\draw (v6) -- (v3);

\draw (v3) -- (v2);
\draw (v4) -- (v2);
\draw (v6) -- (v1);
\draw (v5) -- (v0);
\draw (v1) -- (v3);

\draw (v2) -- (v1);
\draw (v2) -- (v0);

\draw (v1) -- (v0);

\end{scope}

\begin{scope}[xshift=8cm]

\foreach \i in {1,...,5} {
\node (v\i) at ({90 + (\i-1)*72}:1.5cm) {};
}
\foreach \i/\j in {1/2, 2/3, 3/4, 4/5, 5/1} {
\draw (v\i) -- (v\j);
}

\end{scope}

\begin{scope}[xshift=10.5cm]
\node (v1) at (0,0) {};
\node (v2) at (1,0) {};
\node (v3) at (2,0) {};
\node (v4) at (3,0) {};

\draw (v1) -- (v2);
\draw (v2) -- (v3);
\draw (v3) -- (v4);
\end{scope}

\end{tikzpicture}
    \caption{All graphs with eigenvalue sums greater than $n-3$. The first two graphs are complements of each other, and $C_5$ and $P_4$ are self-complementary.}
    \label{fig:placeholder}
\end{figure}



The graphs under 8 vertices whose Nordhaus-Gaddum eigenvalue sum is greater than $n-3$ are shown in Figure 3. These are the only exceptions where both $G$ and $G^c$ are connected where $\mu_2(G)+\mu_2(G^c)>n-3$.  We remark also that for the well-known Petersen graph $P$, which has 10 vertices, we have $\mu_2(P)=2$ and $\mu_2(P^c)=5$, so $\mu_2(P)+\mu_2(P^c)=7=10-3$.  Likewise, for $C_6$, the cycle on $6$ vertices, $\mu_2(C_6)=1$ and $\mu_2(C_6^c)=2$, so $\mu_2(C_6)+\mu_2(C_6^c)=3=6-3$.  So equality in Theorem \ref{thm:alg_conn} is achieved for these examples (see Figure \ref{fig:petersen}).   As we saw in the proof of Theorem \ref{thm:alg_conn}, for cycles $C_n$ we have  $\mu_2(C_n) + \mu(C_n^c) = n-2(\cos(\frac{2\pi}{n})-\cos(\frac{2\pi \lfloor n/2\rfloor}{n}))$.  For $n>6$, this quantity lies strictly between $n-3$ and $n-4$, so cycles are an infinite family of graphs whose Nordhaus-Gaddum sum is close to the bound we have proven.

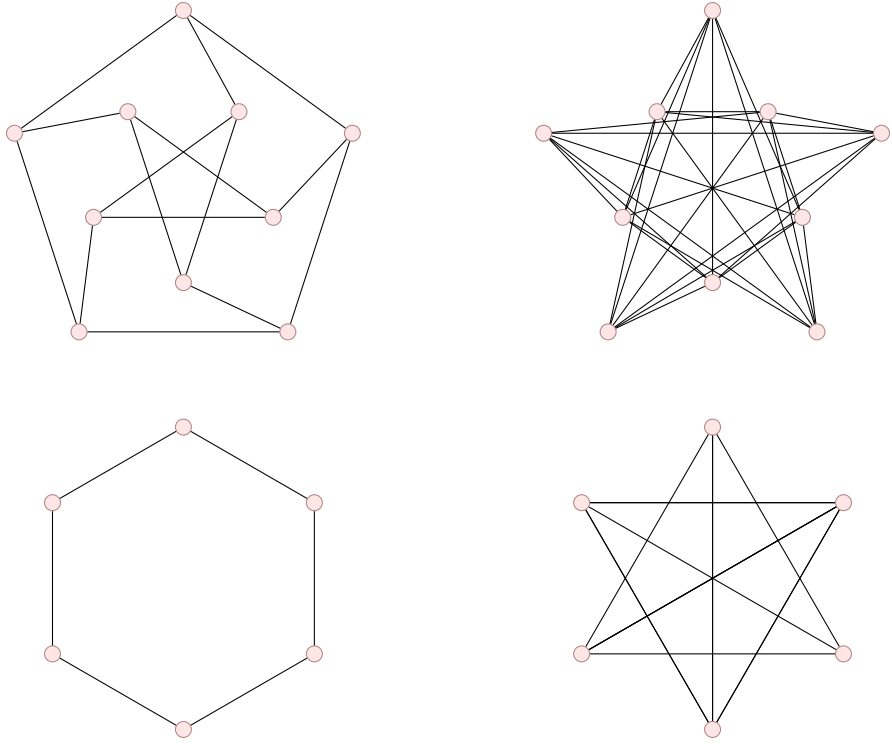
\begin{figure}
    \centering

\begin{tikzpicture}[
every node/.style={circle, fill=pink!40, draw=pink!70!black,
inner sep=0pt, minimum size=6pt},
every edge/.style={gray!70, thick}
]

  \begin{scope}
    \PetersenVertices
    \foreach \a/\b in {
      1/2,2/3,3/4,4/5,5/1,
      1/6,2/7,3/8,4/9,5/10,
      6/8,8/10,10/7,7/9,9/6}
    {
      \draw (v\a)--(v\b);
    }
  \end{scope}

  \begin{scope}[xshift=7cm]
    \PetersenVertices
    \foreach \a/\b in {
      1/3,1/4,1/7,1/8,1/9,1/10,
      2/4,2/5,2/6,2/8,2/9,2/10,
      3/5,3/6,3/7,3/9,3/10,
      4/6,4/7,4/8,4/10,
      5/6,5/7,5/8,5/9,
      6/7,6/10,
      7/8,
      8/9,
      9/10}
    {
      \draw (v\a)--(v\b);
    }
  \end{scope}

\end{tikzpicture}

\vspace{1cm}

\begin{tikzpicture}[
every node/.style={circle, fill=pink!40, draw=pink!70!black,
inner sep=0pt, minimum size=6pt},
every edge/.style={gray!70, thick}
]

  \begin{scope}
    \foreach \i/\ang in
    {1/30,2/90,3/150,4/210,5/270,6/330}{\node (v\i) at (\ang:2) {};}
    \foreach \a/\b in {
      1/2,2/3,3/4,4/5,5/6,6/1}
    {
      \draw (v\a)--(v\b);
    }
    
  \end{scope}

  \begin{scope}[xshift=7cm]
   \foreach \i/\ang in
    {1/30,2/90,3/150,4/210,5/270,6/330}{\node (v\i) at (\ang:2) {};}
    \foreach \a/\b in {
      1/3,1/4,1/5,2/4,2/5,2/6,
      3/5,3/6,3/1,
      4/6,4/1,4/1,
      5/1,5/2,5/3}
    {
      \draw (v\a)--(v\b);
    }
   
  \end{scope}

\end{tikzpicture}
 \caption{The Petersen graph $P$ with its complement, and the 6-cycle $C_6$ with its complement.  These examples achieve $\mu_2(G)+\mu_2(G^c)=n-3$.}
    \label{fig:petersen}
\end{figure}

  The remainder of this section is devoted to understanding families of graphs whose Nordhaus-Gaddum sum is close to the bounds in Lemmas \ref{lem:spread_bound} and \ref{lem:spred_bound_regular}. The next proposition and its corollary indicate that we should focus on regular graphs.

\begin{proposition}
  For a fixed $\delta$ and $n$, the bound 
\[
\sqrt{(\Delta-\delta)^2+2(\Delta+\delta)-\left(2\frac{\Delta+1}{n}\right)^2\left(\frac{\delta-1}{\Delta+1}+1\right)}
\]
is minimized when $\Delta=\delta$.  
\end{proposition} 

\begin{proof}
    For fixed $\delta$ and $n$, taking the derivative with respect to $\Delta$, 
one can see that the only critical point is outside the interval $[\delta,n-1]$, the the minimum occures when $\Delta=\delta$.
\end{proof} 

\begin{corollary}
    For graphs with minimum degree $\delta$, the Nordhaus-Gaddum upper bound is maximized 
    \[\mu_2(G) + \mu_2(G^c) \leq n - \sqrt{(\Delta-\delta)^2 + 2(\Delta+\delta) - \left(2\frac{\Delta+1}{n}\right)^2\left(\frac{\delta-1}{\Delta+1} +1\right)}\]
    when $\Delta=\delta$.
    \label{cor:1}
\end{corollary}

Corollary \ref{cor:1} says $k$-regular graphs have the highest possible Nordhaus-Gaddum upper bound. This leads us to further investigate graphs with very regular structures. Strongly regular graphs are parameterized by $(n, k, \lambda, \mu)$ where $\lambda$ is the number of common neighbors between adjacent nodes and $\mu$ is the number of common neighbors between non-adjacent nodes. It is known that there are only two non-zero Laplacian eigenvalues for a strongly regular graph and there is a formula to compute each of them.

\begin{lemma}
    The non-zero eigenvalues of a strongly regular graphs are 
    \[\mu = k - \frac{1}{2}\left[(\lambda-\mu) \pm\sqrt{(\lambda - \mu)^2 + 4(k-\mu)}\right].
    \]
    \label{lem:srg_eig}
\end{lemma}

\begin{proposition}
    Strongly regular graphs parameterized by $(n, k, \lambda, \mu)$ have an eigenvalue sum of
    \[
    \mu_2(G) + \mu_2(G^c) = n - \sqrt{(\lambda-\mu)^2 +4(k-\mu)}
    \]
\end{proposition}

\begin{proof}
    
    Using Lemma \ref{lem:srg_eig}, we have $\mu_2(G) = k - \frac{1}{2}\left[(\lambda-\mu) +\sqrt{(\lambda - \mu)^2 + 4(k-\mu)}\right]$. The complement of a strongly regular graph is also strongly regular parameterized by $(n, n-1-k, n-2-2k-\mu, n-2k+\lambda)$. Thus
    \[
    \mu_2(G^c) = n-k - \frac{1}{2}\left[(\mu-\lambda)) +\sqrt{(\lambda-\mu)^2 + 4(k-\mu)}\right]
    \]
    Then, 
    \[
    \mu_2(G)+\mu_2(G^c) = k - \frac{1}{2}\left[(\lambda-\mu) +\sqrt{(\lambda - \mu)^2 + 4(k-\mu)}\right] + n-k - \frac{1}{2}\left[(\mu-\lambda) +\sqrt{(\lambda-\mu)^2 + 4(k-\mu)}\right]
    \]
    and thus
    \[
    \mu_2(G)+\mu_2(G^c) = n-\sqrt{(\lambda-\mu)^2 + 4(k-\mu)}.
    \]
\end{proof}

Analyzing this sum, it is largest when $\lambda$ and $\mu$ are large and equal. Comparing it to the known bound above, when $\lambda$ and $\mu$ are equal, it is on the order of $n-\sqrt{4k}$, which is fairly close to our bound. It is known that if a strongly regular graphs exists then the equality $k(k-\lambda-1)=(n-k-1)\mu$ is satisfied. 

For example, conference graphs are strongly regular graphs parameterized only by n as $(n, \frac{n-1}{2}, \frac{n-5}{4}, \frac{n-1}{4})$. Using the formula for the upper bound, the eigenvalue sum of a conference graph is $n-\sqrt{n}$ which is quite close (and asymptotically equal) to the bound in Lemma \ref{lem:spred_bound_regular} for this regularity.  Thus, if it is desirable to have a graph and complement that are both extremely well connected, conference graphs will be near optimal. 


\section{Further Directions}

In this section we will investigate data for two more graph connectivity parameters and conjecture Nordhaus-Gaddum type upper bounds based on this data. We will examine the Cheeger constant and the normalized Laplacian. Much like how we did previously, we analyze the Cheeger constant to give us intuition of the normalized Laplacian matrix. As mentioned in the introduction, there is a famous result linking these two invariants called the Cheeger inequality.

\begin{lemma}[Lemma 2.1 and Theorem 2.2 of \cite{Fanchungbook}]
    For any graph $G$, \[\frac{h(G)^2}{2}\leq\lambda_2(G)\leq 2h(G).\]
\end{lemma}

\subsection{Cheeger Constant}

The Cheeger constant can be thought of as a the normalized version of the isoperimetric number. It is defined as 
\[h(G) = \min_{X\subset V}\frac{|\partial(X, \bar{X})|}{\min\{vol(X), vol(\bar{X})\}}.\]

We will identify families of graphs that performed well experimentally and conjecture what a Nordhaus-Gaddum upper bound would be for the Cheeger constant. In Figure \ref{cheegerplots}, we plotted Cheeger constant sums for a graph and its complement up to 9 vertices as we have done before.

\begin{figure}
    \includegraphics[width=0.3\linewidth]{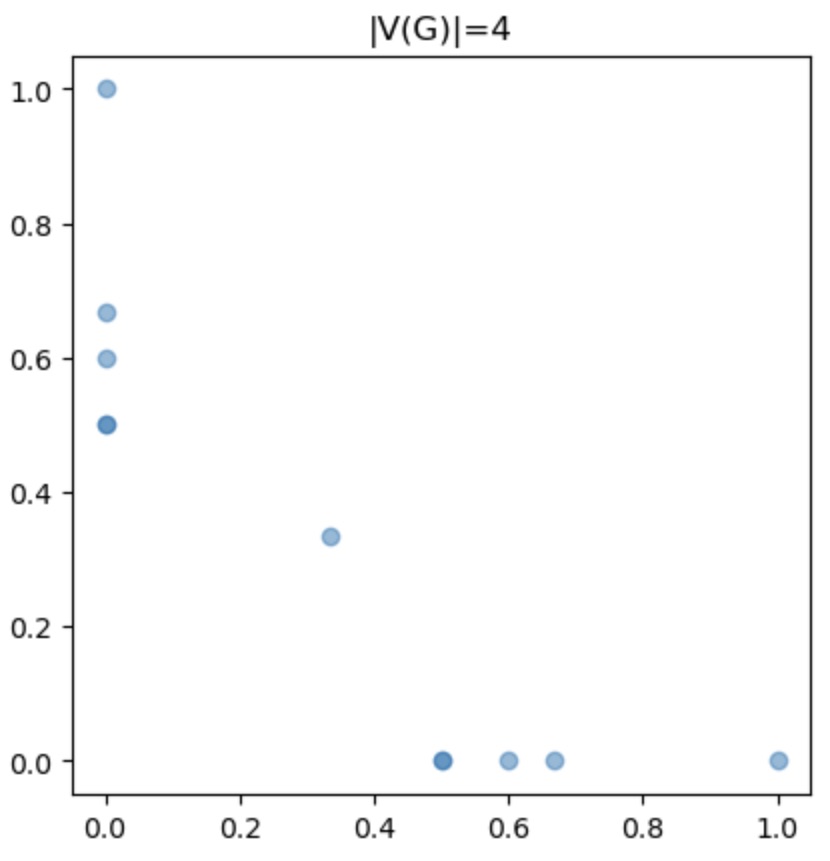}
    \includegraphics[width=0.3\linewidth]{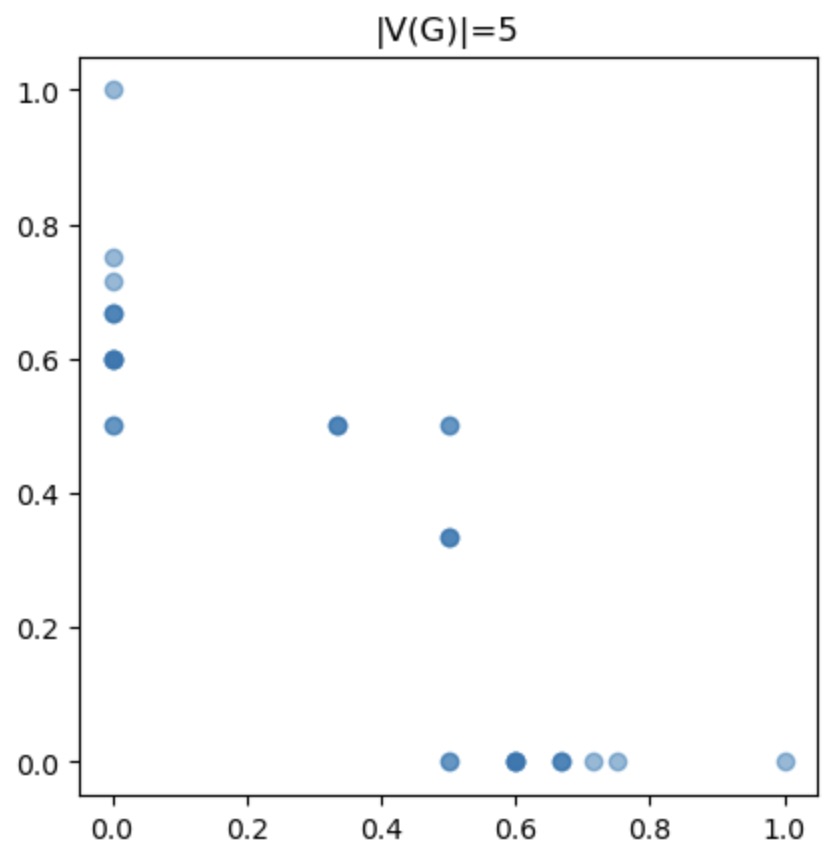}
    \includegraphics[width=0.3\linewidth]{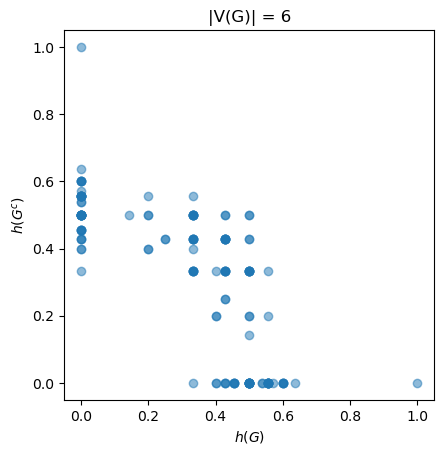}
        \centering

    \centering
    \includegraphics[width=0.3\linewidth]{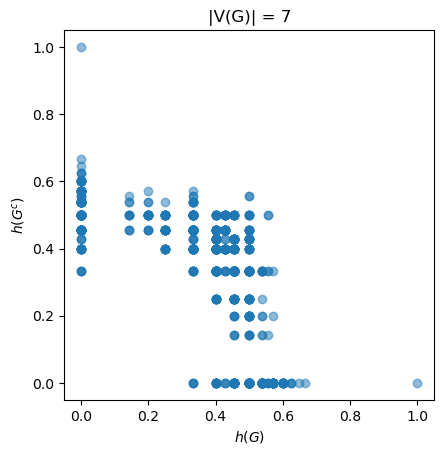}
    \includegraphics[width=0.3\linewidth]{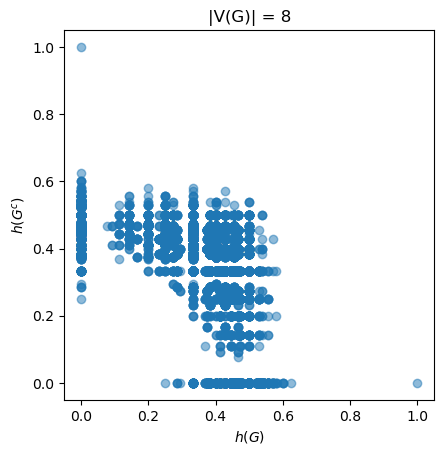}
    \includegraphics[width=0.3\linewidth]{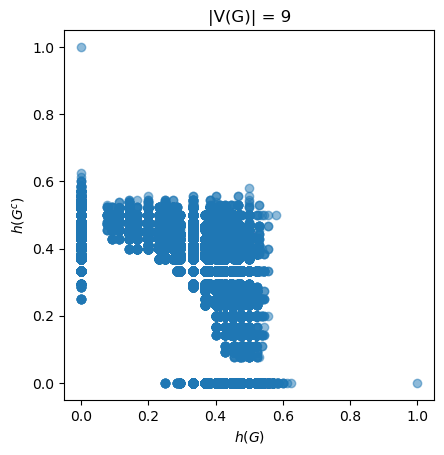}

    \caption{Plots for Cheeger Constant for $n=4, 5, 6, 7, 8, 9$}\label{fig:cheeger_plots}
    \label{cheegerplots}
\end{figure}

We make some useful observations to ease the computations regarding the Cheeger constant. First, note that $vol(X)=\partial(X,\bar X)$+2$\partial(X,X)$. Additionally, if a graph is disconnected, then the Cheeger constant is 0. 
For simplicity, we will use the notation $\partial X$ to refer to the edges between the vertices contained in $X$. That is, $\partial X=\partial(X,X)$.

\begin{lemma}
    If $G$ does not have a vertex with degree $n-1$ and $G$ is connected, then there exists a set $X$ such that $\partial X$ and $\partial \bar X$ are both non-zero.
    \label{cheeger_nonzero}
\end{lemma}

\begin{proof}
    By way of contradiction, we assume that no such set $X$ exists. Then there is a set $X$ such that either $\partial X$ or $\partial \bar X$ will have to be zero. Without loss of generality, we assume $\partial X=0$. Then we see that since $G$ is connected, every vertex in $X$ must be connected with at least one edge in $\bar X$. Now, we choose a vertex $y \in X$ and $z \in \bar X$, such that $y$ and $z$ are connected. Then we see that for $G$ to be connected, either $y$ would have to be connected to another vertex in $\bar X$ or $z$ would have to be connected to a vertex in $X$ or $\bar X$. Now, we choose the set $Y$=\{$y,z$\}, Then since every vertex in $X$ is connected to another vertex in $\bar X$, we see that the set $\partial \bar Y$ is non-zero. But  $\partial Y$ is also non-zero, which is a contradiction to our initial assumption. 
\end{proof}

\begin{proposition}
    $h(G)\leq1$ with equality if and only if $G=S_n$.
\end{proposition}

\begin{proof}
    By definition of Cheeger constant and Lemma \ref{cheeger_nonzero}, we see that if a graph is connected and none of the vertices have degree $n-1$, then the Cheeger constant is strictly less than 1. If the graph is disconnected, then $h(G)=0$. So we focus on the equality condition by focusing on graphs that have at least one vertex with degree $n-1$.  
    
    ($\implies)$ We show the contrapositive of the statement, namely that if $G\neq S_n$, then $h(G)<1$. We have that the graph $G$ has at least one vertex with degree $n-1$, and since $G\neq S_n$, we see that there have to be at least two vertices $x$ and $y$ with an edge between them. Then for a set $X$=\{$x,y$\}, we see that $h(G)=\frac{\partial(X,\bar X)}{\partial(X,\bar X)+2\partial  X}=\frac{\partial(X,\bar X)}{\partial(X,\bar X)+2}<1$. So we see that if $G$ has at least 1 vertex with degree $n-1$, but is not the star graph, then $h(G)<1$.  
    
    ($\impliedby)$ Now, we assume that $G=S_n$, meaning $G$ has exactly one vertex $y$ with degree $n-1$ and every other vertex is connected with just $y$ and not with any other vertex. We deal with two cases.  
    
    Case 1: The set $X$ does not have the dominating vertex. Then we see that every vertex that is contained in $X$ can only be connected to the dominating vertex and no edge can exist between any number of vertices contained in the set $X$. This means that $\partial X=0$, so $h(G)=\frac{\partial(X,\bar X)}{\partial(X,\bar X)+\partial X}=\frac{\partial(X,\bar X)}{\partial(X,\bar X)}=1$. 
    
    Case 2: The set $X$ includes the dominating vertex. Now, we see that the dominating vertex has degree $n-1$ and the remaining $n-1$ vertices all have degree 1, meaning adding even one extra vertex in $X$ on top of the dominating vertex results in the set $\bar X$ containing the smaller volume. Since $\bar X$ does not contain the dominating vertex, we see that $h(G)=1$ by Case 1. If $X$ contains just the dominating vertex, then we see that $\partial (X,\bar X)=n-1$ and $\partial X=0$, so substituting these values in the formula for $h(G)$ yields $h(G)=1$. 
\end{proof}
\begin{figure}
    \centering
    \begin{tikzpicture}[
every node/.style={circle, fill=pink!40, draw=pink!70!black,
inner sep=0pt, minimum size=6pt},
every edge/.style={gray!70, thick}
]
\begin{scope}
\node (center) at (0,0) {};

\foreach \i in {1,...,5} {
  \node (v\i) at ({90 + (\i-1)*72}:1.5cm) {};
  \draw (center) -- (v\i);
}
\end{scope}

\begin{scope}[xshift=4cm]

\node (0) at (.5,0) {};
\node (1) at (-.5,0) {};

\node (2) at (1.3,-1) {};
\node (3) at (-1.3,-1) {};
\node (4) at (-1.5,.7) {};
\node (5) at (-.5,1.5) {};
\node (6) at (1.5,.7) {};
\node (7) at (0,-1.5) {};
\node (8) at (.5,1.5) {};

\draw (0) -- (1);
\draw (0) -- (2);
\draw (0) -- (4);
\draw (0) -- (3);
\draw (0) -- (5);
\draw (0) -- (6);
\draw (0) -- (7);
\draw (0) -- (8);

\draw (1) -- (2);
\draw (1) -- (3);
\draw (1) -- (4);
\draw (1) -- (5);
\draw (1) -- (6);
\draw (1) -- (7);
\draw (1) -- (8);

\end{scope}

\begin{scope}[xshift=8cm]
\node (2) at (0,1) {};
\node (0) at (-.5,0) {};
\node (1) at (.5,0) {};

\node (4) at (-1.25,1.5) {};
\node (6) at (1.25,1.5) {};
\node (3) at (-1.5,-1) {};
\node (5) at (1.5,-1) {};

\draw (2) -- (0);
\draw (2) -- (1);
\draw (2) -- (4);
\draw (2) -- (6);
\draw (2) -- (3);
\draw (2) -- (5);

\draw (0) -- (1);
\draw (0) -- (3);
\draw (0) -- (4);
\draw (0) -- (5);
\draw (0) -- (6);

\draw (1) -- (5);
\draw (1) -- (6);
\draw (1) -- (3);
\draw (1) -- (4);

\end{scope}

\end{tikzpicture}
\caption{The star graph on 6 vertices, $S_6$; The generalized stars, $K_2 \vee E_7$ and $K_3 \vee E_4$}
\end{figure}
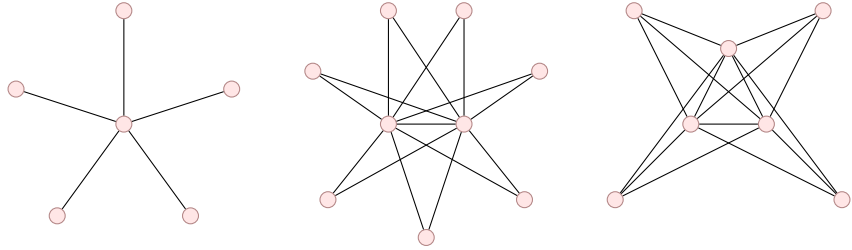


Since the $S_n$ has a vertex with degree $n-1$, its complement is disconnected. We identify graphs similar to $S_n$ to find graphs with high Cheeger constants. Notably, we look at graphs with very highly connected vertices and adjacent lower degree vertices. We identified one particular family we call the generalized star graph. It has the form $G=K_n \vee E_s$, where $\vee$ means the join of these two graphs so every vertex in $E_s$ is adjacent to every vertex in $K_n$. From this notation, $r+s = n$. Each vertex in $K_n$ is connected to every other vertex in the graph so $G^c$ is disconnected. In deriving our formula, we will use our parameters of $n$, $r$, and $s$ related to our set $X$ as parameters. By definition, the case $r=1$ is the star graph, the case $r=0$ is the empty graph $E_n$, and the case $s=1$ is $K_n$.



        From observation, $K_3 \vee E_{n-3}$ had the highest Cheeger constant among the generalized star graph family. As can be seen in the condition when $n \mod 3=0$, this tends to $\frac{5}{9}$ as $n \rightarrow\infty$. We classify these graphs separately. When $G=K_3 \vee E_{n-3}$, we see that 
    
    \[h(G)=\begin{cases}
      \frac{2 + \lfloor\frac{n-3}{3}\rfloor +2\lceil\frac{2(n-3)}{3}\rceil}{n-1 + 3(\lceil\frac{2(n-3)}{3}\rceil)} & \text{if $n$ mod 3 = 1.}\\
      \frac{2 + \lceil\frac{n-3}{3}\rceil +2\lfloor\frac{2(n-3)}{3}\rfloor}{n-1 + 3\lfloor\frac{2(n-3)}{3}\rfloor} & \text{if $n$ mod 3 = 2.}\\
      \frac{5n-9}{9n-21} & \text{if $n$ mod 3 = 0.}
    \end{cases}\]



    There is another family of graphs that tended to yield higher Cheeger constants. This was the graph $G$ which consisted of the complete graph on $n-1$ vertices with the remaining vertex connected to exactly one vertex on the complete graph. This graph has a Cheeger constant that can be denoted by the formula $h(G)=\frac{\lceil\frac{n}{2}\rceil \lfloor\frac{n-2}{2}\rfloor}{n+\lfloor\frac{n-4}{2}\rfloor(n-2)}$.

    Looking at the data in Figure \ref{fig:cheeger_plots}, the natural Nordhaus-Gaddum question seems to be $\max \{h(G),h(G^c)\}$. Ignoring $S_n$, the plots show a box where all graphs not equal to $S_n$ are bounded within some of these high performing graphs. The structure of the plots lends itself better to this type of upper bound. $\max \{h(G),h(G^c)\}$ gives us the value of the edges of this box. We see that Cheeger constant on this condition could ideally be bounded by the various families that we were able to work with and based off the formulas that we were able to come up. Considering the high results $K_3 \vee E_{n-3}$, we make the following conjecture:
    \begin{conjecture}
        If $G \neq S_n$, then $\max \{h(G),h(G^c)\} \leq h(K_3 \vee E_{n-3})$.
    \end{conjecture}
    

\subsection{Normalized Laplacian Eigenvalues}

With the Cheeger constant offering leads for the normalized Laplacian matrix, we now turn our attention to $\lambda_2(G)$. In Figure \ref{fig:lapl_plots}, we create a plot for normalized Laplacian eigenvalue Nordhaus-Gaddum problem. We focus first on $\max\{\lambda_2(G), \lambda_2(G^c)\}$.

\begin{figure}[h]
    \includegraphics[width=0.3\linewidth]{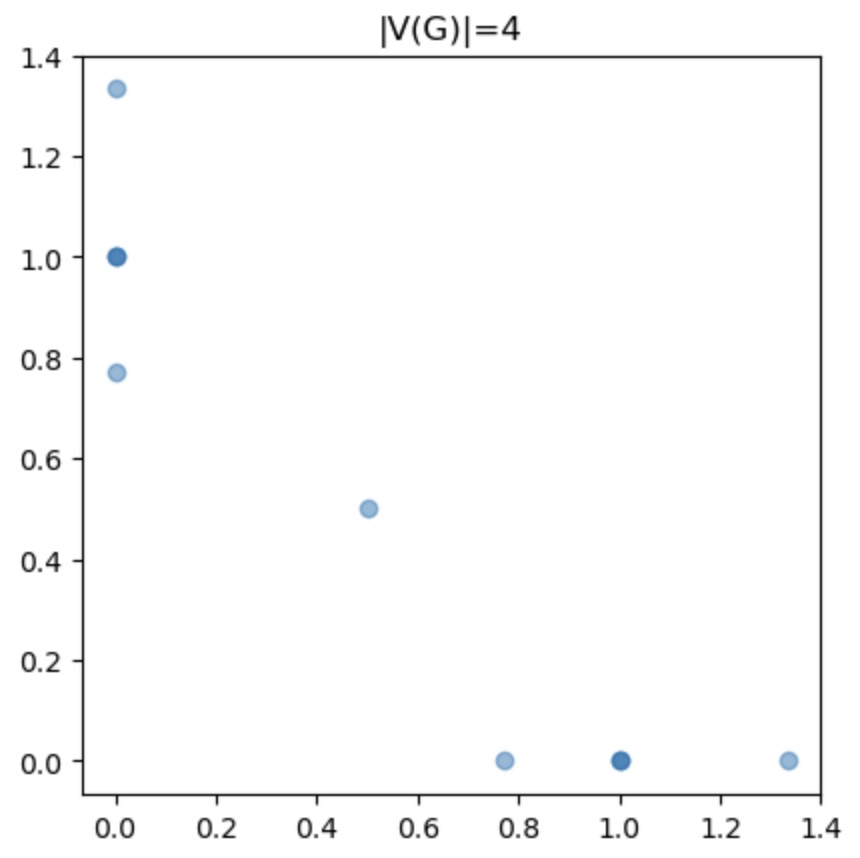}
        \centering
    \includegraphics[width=0.3\linewidth]{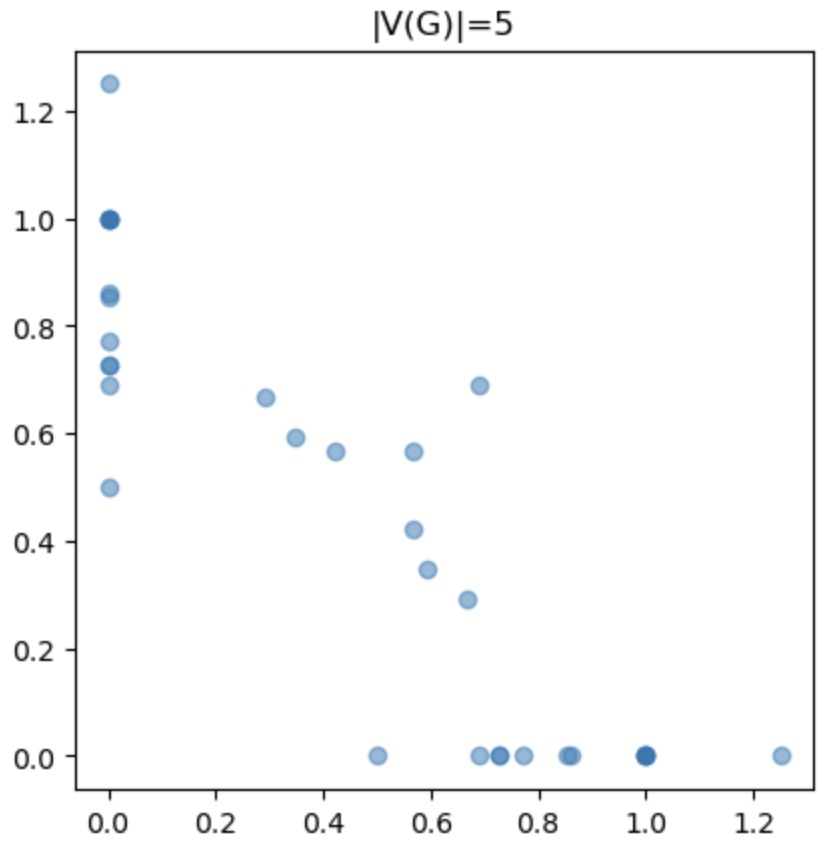}
    \includegraphics[width=0.3\linewidth]{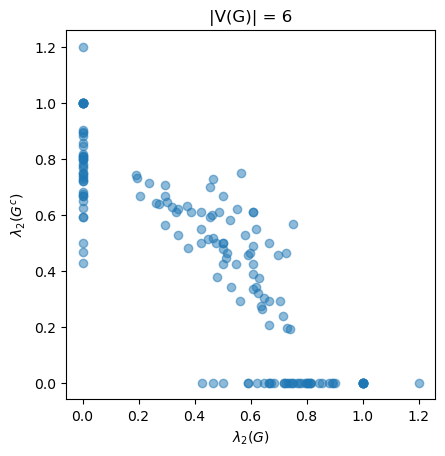}
    
    \includegraphics[width=0.3\linewidth]{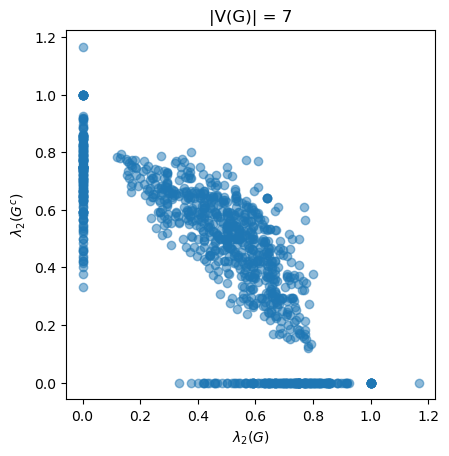}
    \centering
    \includegraphics[width=0.3\linewidth]{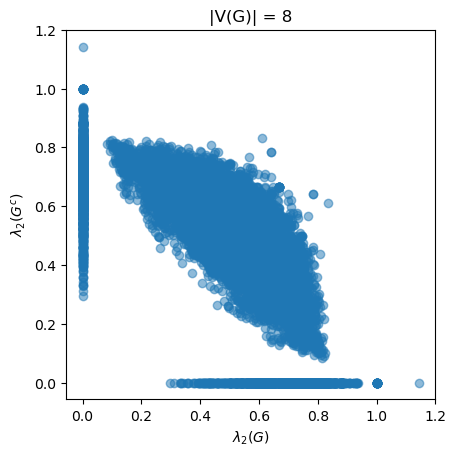}
    \includegraphics[width=0.3\linewidth]{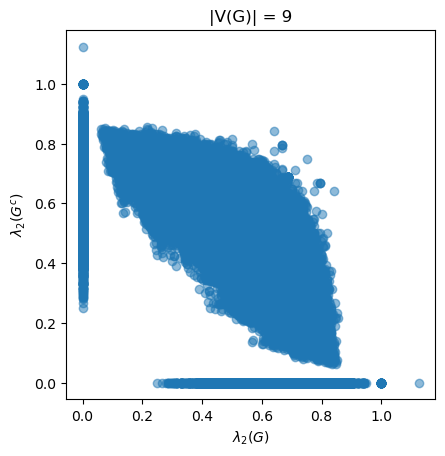}
    \caption{Plots of $\lambda_2(G) + \lambda_2(G^c)$ up to isomorphism for $n=4,5,6,7,8,9$}\label{fig:lapl_plots}
\end{figure}

\begin{proposition}
        If $G=K_n$, then $\lambda_2(G)=\frac{n}{n-1}$. If $G \neq K_n$, then $\lambda_2(G) \leq 1$.
\end{proposition}

    \begin{proof}
        The proof is outlined in \cite{Fanchungbook}.
    \end{proof}


    
    Using our results from the Cheeger constant, we identified some of the graphs where $\lambda_2(G) =1$. The proof of the following we leave to the appendix. 
    
    \begin{proposition} The normalized Laplacian spectrum of the Generalized Star Graph is $0$, $1$, $\frac{n}{n-1}, \frac{s+n-1}{n-1}$ with multiplicity $1$, $r-1$, $s-1$, and $1$ respectively. \label{prop:gen_star}
    \end{proposition}


 We will now look at graphs that appear to have large $\lambda_2(G) + \lambda_2(G^c)$.
 This yielded some interesting families with the potential to yield good bounds on this condition. We will analyze a certain family on odd numbers of vertices and a family of graphs on even numbers of vertices. From these two families, we identify potential Nordhaus-Gaddum bounds that can be conjectured.

For an odd number of vertices, we look at the $\frac{n-1}{2}$-regular self complementary graphs. 
Of the graphs we analyzed experimentally, the Paley graph performed the best. The Paley graph is a strongly regular graph parameterized by $(n, \frac{n-1}{2}, \frac{n-5}{4}, \frac{n-1}{4})$. By being self-complementary, $\lambda_2(G) = \lambda_2(G^c)$ and using the formulas around strongly regular graphs, we know $\lambda_2(G) = \frac{n-\sqrt{n}}{n-1}$. Thus, $\lambda_2(G)+\lambda_2(G^c)= \frac{2(n-\sqrt{n})}{n-1}$. Looking at its vertex long term behavior $n \rightarrow \infty$, $\lambda_2(G)+\lambda_2(G^c)$ converges to 2. This aligns with what we understood with the Laplacian Nordhaus-Gaddum upper bound. The normalized Laplacian normalizes each entry of the Laplacian matrix by the degree of the shared vertices. So when looking at the strongly regular graph $G$, its eigenvalues will be normalized by the degree. $G^c$ will be strongly regular as well but without loss of generality if the degree is high in $G$ and low in $G^c$, then $\lambda_2(G)$ will be more penalized and $\lambda_2(G^c)$ will have a higher growth.

We now define a family of graphs with even numbers of vertices that appears to be optimal as follows: $G$ is composed of a complete graph of $\frac{n}{2}$ vertices and the remaining $\frac{n}{2}$ vertices were each connected to exactly 1 vertex in the complete graph. The complement also consisted of a complete graph on $\frac{n}{2}$ vertices, but each of the remaining $\frac{n}{2}$ vertices are connected to all but 1 vertex on the complete graph, with no two vertices being connected to the exact same set of vertices. See Figure \ref{fig:optimal6}.

        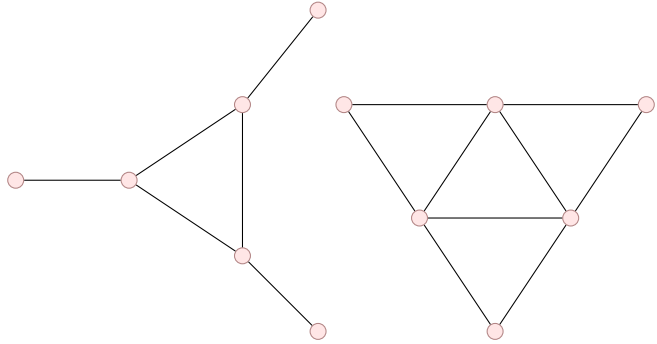
\begin{figure}
        \centering
            \begin{tikzpicture}[
every node/.style={circle, fill=pink!40, draw=pink!70!black,
inner sep=0pt, minimum size=6pt},
every edge/.style={gray!70, thick}
]
        
        \node (0) at (3, 1)   {};
        \node (1) at (1.5, 2) {};
        \node (2) at (3, 3)   {};
        \node (3) at (4, 0) {};
        \node (4) at (0, 2)   {};
        \node (5) at (4, 4.25) {};
        
        \draw (0) -- (1);
        \draw (0) -- (2);
        \draw (1) -- (2);
        \draw (0) -- (3);
        \draw (1) -- (4);
        \draw (2) -- (5);
        
        \end{tikzpicture}
          \begin{tikzpicture}[
every node/.style={circle, fill=pink!40, draw=pink!70!black,
inner sep=0pt, minimum size=6pt},
every edge/.style={gray!70, thick}
]

\node (0) at (4, 4)  {};
\node (1) at (2, 1)  {};
\node (2) at (0, 4)  {};
\node (3) at (1, 2.5){};
\node (4) at (2, 4)  {};
\node (5) at (3, 2.5)    {};

\draw (0) -- (4);
\draw (0) -- (5);
\draw (1) -- (3);
\draw (1) -- (5);
\draw (2) -- (3);
\draw (2) -- (4);
\draw (3) -- (4);
\draw (3) -- (5);
\draw (4) -- (5);

\end{tikzpicture}
          \caption{The focus graph (left) and its complement on 6 vertices}
          \label{fig:optimal6}
        \end{figure}
        
        We focus on the spectrum on these graphs. We see that the normalized Laplacian matrices for these graphs are of the form 
        \[\mathcal{L}(G)=\begin{bmatrix}
            I_{\frac{n}{2}} && -\sqrt{\frac{2}{n}}I_{\frac{n}{2}}   \\
            -\sqrt{\frac{2}{n}}I_{\frac{n}{2}} && -\frac{2}{n}J_{\frac{n}{2}}+\frac{n+2}{n}I_{\frac{n}{2}}  \\
        \end{bmatrix} \]
        \[\mathcal{L}(G^c)=\begin{bmatrix}
            -\frac{1}{n-2}J_{\frac{n}{2}}+\frac{n-1}{n-2}I_{\frac{n}{2}} && -\frac{\sqrt{2}}{n-2}(J_{\frac{n}{2}}-I_{\frac{n}{2}})  \\
            -\frac{\sqrt{2}}{n-2}(J_{\frac{n}{2}}-I_{\frac{n}{2}}) && I_{\frac{n}{2}}   \\
        \end{bmatrix}\]
         We compute the spectrum for both these matrices separately. 
\begin{proposition}
    The eigenvalues of the focus graph $G$ are 0, $\frac{n+1-\sqrt{2n+1}}{n}$, $\frac{n+1+\sqrt{2n+1}}{n}$, $\frac{n+2}{n}$ with multiplicity 1, $\frac{n-2}{2}$, $\frac{n-2}{2}$, and 1 respectively. For $G^c$, the eigenvalues are 0, $\frac{n-3}{n-2}$, $\frac{n}{n-2}$, $\frac{3}{2}$ with multiplicities 1, $\frac{n-2}{2}$, $\frac{n-2}{2}$, and 1.
    \label{prop:focus_G}
\end{proposition}

\begin{proof}
    The proof is similar to Proposition \ref{prop:gen_star} and so we leave the technical details to the appendix.
\end{proof}
        
    Looking at this family, we see $\lambda_2(G) = \frac{n+1-\sqrt{2n+1}}{n}$ and $\lambda_2(G^c) = \frac{n-3}{n-2}$. Thus,  $\lambda_2(G)+\lambda_2(G^c)$=$\frac{n+1-\sqrt{2n+1}}{n}+\frac{n-3}{n-2}$ is the position at which a maximum is achieved on graphs with even degrees. As $n$ gets larger, we see that both $\lambda_2(G)$ and $\lambda_2(G^c)$ approach 1. However, it can be seen $\lambda_2(G)$ and $\lambda_2(G^c)$ will never equal 1.
    Combining both, we see that $\lambda_2(G)+\lambda_2(G^c)<2$, and converges to 2 as $n\rightarrow\infty$.

    We summarize our discussion in the following conjecture.

    \begin{conjecture}
        If $G$ has $n$ vertices, than 
        \[\lambda_2(G)+\lambda_2(G^c)\leq 2-O\left(\frac{1}{\sqrt{n}}\right).\]
    \end{conjecture}

    Finally, we remark that many of the graphs we have found that appear to have the largest Nordhaus-Gaddum sum for the normalized Laplacian are not regular, which is in stark contract to the case with the combinatorial Laplacian.

\section{Appendix}

Here are the technical details of Proposition \ref{prop:gen_star}.
\begin{proof}
        The normalized Laplacian of the generalized star graph in block form be
    \[ M = \begin{bmatrix}
    \frac{n}{n-1}I_{r} - \frac{1}{n-1}J_{r} & \frac{-1}{\sqrt{(n-1)r}}J_{r\times s} \\
    \frac{-1}{\sqrt{(n-1)r}}J_{s\times r} & I_{s}
    \end{bmatrix}.\]
    The all 1's vector yields the eigenvalue 0.  \\
     Now consider any vector $x = \begin{bmatrix}
         0_r  \\
         x_{r+1}  \\
         .  \\
         .  \\
         x_{r+s}
     \end{bmatrix}$, such that $\sum^{r+s}_{r+1}x_i$=0. \\ Then 
     $Mx$ = $\begin{bmatrix}
         \sum_{r+1}^n\frac{-1}{\sqrt{(n-1)r}}x_i,
         , \dots,
         \sum_{i=r+1}^n\frac{-1}{\sqrt{(n-1)r}}x_i, 
         x_{r+1}, \dots
         x_{r+s}
     \end{bmatrix}^T=\begin{bmatrix}
         0, 0,\dots,
         x_{r+1},\dots,
         x_{r+s}
     \end{bmatrix}^T=x$  \\
     Thus the associated eigenvalue is 1. Since the space of such vectors has dimension $s-1$, the eigenvalue 1 has multiplicity $s-1$.  \\
     Now consider vector $y = \begin{bmatrix}
         y_1  \\
         y_2  \\
         .  \\
         .  \\
         .  \\
         y_r  \\
         0_s
     \end{bmatrix}$, where $\sum^r_1 y_i$=0. \\ Then 
    $My = \begin{bmatrix}
        \frac{n}{n-1}y_1+\sum_{1}^{r}\frac{-1}{n-1}y_{j},\dots,\frac{n}{n-1}y_r+\sum_{1}^{r}\frac{-1}{n-1}y_{j}, 0_s\end{bmatrix}^T$=$\frac{n}{n-1}\begin{bmatrix}
        y_1,\dots,
        y_r,
        0_s
    \end{bmatrix}^T=\frac{n}{n-1}y$.
    
    In this case, we see that the multiplicity for the eigenvalue $\frac{n}{n-1}$ will be $r-1$.  
    
    There is one more eigenvalue left, which we will compute using the trace. 
    Now, $tr(M)=\sum_1^n \lambda_i=1(s-1)+\frac{n}{n-1}(r-1)+\lambda$. We know that $tr(M)=n$. This shows that $\lambda_j=n-1(s-1)-\frac{n}{n-1}(r-1)=\frac{(n-s+1)(n-1)-n(r-1)}{n-1}=\frac{(n-s+1)(n-1)-n(n-s-1)}{n-1}=\frac{n+s-1}{n-1}$. So we see that the final eigenvalue is $\frac{n+s-1}{n-1}$ with multiplicity 1.  
    
    From the spectrum, we see that the algebraic connectivity of the generalized star graph is $\lambda_{2} = 1$.
\end{proof}

Here are the technical details of the proof for Proposition \ref{prop:focus_G}.

\begin{proof}
    As 0 belongs in the spectrum of both $\mathcal{L}(G)$ and $\mathcal{L}(G^c)$, we focus on computing the remaining eigenvalues. First we compute the spectrum on \[\mathcal{L}(G)=\begin{bmatrix}
            I_{\frac{n}{2}} && -\sqrt{\frac{2}{n}}I_{\frac{n}{2}}   \\
            -\sqrt{\frac{2}{n}}I_{\frac{n}{2}} && -\frac{2}{n}J_{\frac{n}{2}}+\frac{n+2}{n}I_{\frac{n}{2}}  \\
        \end{bmatrix} \]
        \[\mathcal{L}(G^c)=\begin{bmatrix}
            -\frac{1}{n-2}J_{\frac{n}{2}}+\frac{n-1}{n-2}I_{\frac{n}{2}} && -\frac{\sqrt{2}}{n-2}(J_{\frac{n}{2}}-I_{\frac{n}{2}})  \\
            -\frac{\sqrt{2}}{n-2}(J_{\frac{n}{2}}-I_{\frac{n}{2}}) && I_{\frac{n}{2}}   \\
        \end{bmatrix}\]
    We let $x_n$=$\frac{1}{n}(\sqrt{\frac{n}{2}}+\sqrt{n^2+\frac{n}{2}})$ and choose a vector $x=\begin{bmatrix}
            -x_n                \\
             x_n                \\
             0_{\frac{n}{2}-2}  \\
            -1                  \\
             1                  \\
             0_{\frac{n}{2}-2}
        \end{bmatrix}$.  
    Then we see that  \\
    $\mathcal{L}(G)$x=$\begin{bmatrix}
            I_{\frac{n}{2}}\begin{bmatrix}
                -x_n \\
                 x_n \\
                 0_{\frac{n}{2}-2}
            \end{bmatrix} && -\sqrt{\frac{2}{n}}I_{\frac{n}{2}}\begin{bmatrix}
                -1  \\
                1   \\
                0_{\frac{n}{2}-2}
            \end{bmatrix}   \\
            -\sqrt{\frac{2}{n}}I_{\frac{n}{2}}\begin{bmatrix}
                -x_n  \\
                 x_n  \\
                 0_{\frac{n}{2}-2}
            \end{bmatrix} && -\frac{2}{n}J_{\frac{n}{2}}+\frac{n+2}{n}I_{\frac{n}{2}} \begin{bmatrix}
                -1  \\
                 1  \\
                 0_{\frac{n}{2}-2}
            \end{bmatrix}  \\
    \end{bmatrix}$=
    $\begin{bmatrix}
            \begin{bmatrix}
             -x_n  \\
              x_n  \\
              0_{\frac{n}{2}-2}
        \end{bmatrix}+\begin{bmatrix}
            \sqrt{\frac{2}{n}}  \\
           -\sqrt{\frac{2}{n}}  \\
            0_{\frac{n}{2}-2} 
        \end{bmatrix}  \\
        \begin{bmatrix}
             \sqrt{\frac{2}{n}}x_n  \\
            -\sqrt{\frac{2}{n}}x_n  \\
            0_{\frac{n}{2}-2}
        \end{bmatrix}+\begin{bmatrix}
            -1-\frac{2}{n}  \\
             1+\frac{2}{n}  \\
             0_{\frac{n}{2}-2}
        \end{bmatrix}
    \end{bmatrix}$\\=
    $\begin{bmatrix}
        -x_n+\sqrt{\frac{2}{n}}  \\
         x_n-\sqrt{\frac{2}{n}}  \\
         0_{\frac{n}{2}-2}       \\
         \sqrt{\frac{2}{n}}x_n-1-\frac{2}{n}  \\
        -\sqrt{\frac{2}{n}}x_n+1+\frac{2}{n}  \\
        0_{\frac{n}{2}-2}
    \end{bmatrix}$=$\begin{bmatrix}
        \frac{1}{\sqrt{2n}}-\frac{\sqrt{2n+1}}{\sqrt{2n}}  \\
       -\frac{1}{\sqrt{2n}}+\frac{\sqrt{2n+1}}{\sqrt{2n}}  \\
       0_{\frac{n}{2}-2}  \\
       -1-\frac{1}{n}+\frac{1}{n}\sqrt{2n+1}  \\
        1+\frac{1}{n}-\frac{1}{n}\sqrt{2n+1}  \\
        0_{\frac{n}{2}-2}
    \end{bmatrix}$=$\frac{n+1-\sqrt{2n+1}}{n}\begin{bmatrix}
        -x_n  \\
         x_n  \\
         0_{\frac{n}{2}-2}  \\
         -1  \\
          1  \\
          0_{\frac{n}{2}-2}
    \end{bmatrix}$=$\frac{n+1-\sqrt{2n+1}}{n}x$.  \\
    We see that this yields the eigenvalue $\frac{n+1-\sqrt{2n+1}}{n}$. For similar vectors obtained by swapping the location of $x_n$ and the concurrent 1, we see that this eigenvalue will have a multiplicity $\frac{n-2}{2}$. Now if we set $x_n$=$\frac{1}{n}(\sqrt{\frac{n}{2}}-\sqrt{n^2+\frac{n}{2}})$ and along the same vectors, this yields an eigenvalue $\frac{n+1+\sqrt{2n+1}}{n}$ with multiplicity $\frac{n-2}{2}$, leaving just one unknown value in the spectrum of $\mathcal{L}(G)$. We set this unknown value as $\lambda$. Now, tr($\mathcal{L}(G))=n$. Since tr(G)=$\sum_{i=1}^{n} \lambda_i$, we see that $n=\frac{(n+1)(n-2)}{n}$+$\lambda$ (the positive and negative $\sqrt{2n+1}$ will cancel out). Then we have $\lambda=n-\frac{(n+1)(n-2)}{n}=n-\frac{n^2-n-2}{n}$=$\frac{n+2}{n}$. So $\frac{n+2}{n}$ is also present in the spectrum with multiplicity 1.  From the spectrum, we see that for this family of graphs, $\lambda_2(G)$=$\frac{n+1-\sqrt{2n+1}}{n}$. Now we will focus on the spectrum of the complement.  \\
    Now, we choose a vector x=$\begin{bmatrix}
         \frac{1}{\sqrt{2}}  \\
        -\frac{1}{\sqrt{2}}  \\
        0_{\frac{n}{2}-2}    \\
        -1  \\
         1\\
         0_{\frac{n}{2}-2}
    \end{bmatrix}$ \\
    $\mathcal{L}(G^c)$x=$\begin{bmatrix}
            (-\frac{1}{n-2}J_{\frac{n}{2}}+\frac{n-1}{n-2}I_{\frac{n}{2}})\begin{bmatrix}
                 \frac{1}{\sqrt{2}}  \\
                -\frac{1}{\sqrt{2}}  \\
                0_{\frac{n}{2}-2}
            \end{bmatrix}  -\frac{\sqrt{2}}{n-2}(J_{\frac{n}{2}}-I_{\frac{n}{2}})\begin{bmatrix}
                -1  \\
                 1  \\
                 0_{\frac{n}{2}-2}
            \end{bmatrix}  \\
            -\frac{\sqrt{2}}{n-2}(J_{\frac{n}{2}}-I_{\frac{n}{2}})\begin{bmatrix}
                 \frac{1}{\sqrt{2}}  \\
                -\frac{1}{\sqrt{2}}  \\
                 0_{\frac{n}{2}-2}  \\
            \end{bmatrix}+I_{\frac{n}{2}}\begin{bmatrix}
                -1  \\
                 1  \\
                 0_{\frac{n}{2}-2}
            \end{bmatrix}
            \end{bmatrix}$=$\begin{bmatrix}
                \begin{bmatrix}
                    \frac{1}{\sqrt{2}}+\frac{1}{(n-2)\sqrt{2}}  \\
                   -\frac{1}{\sqrt{2}}-\frac{1}{(n-2)\sqrt{2}}  \\
                   0_{\frac{n}{2}-2}
                \end{bmatrix}+\begin{bmatrix}
                    -\frac{\sqrt{2}}{n-2}  \\
                     \frac{\sqrt{2}}{n-2}  \\
                     0_{\frac{n}{2}-2}  \\
                \end{bmatrix}  \\
                \begin{bmatrix}
                    \frac{1}{n-2}  \\
                   -\frac{1}{n-2}  \\
                    0_{\frac{n}{2}-2}  \\
                \end{bmatrix}+\begin{bmatrix}
                    -1  \\
                     1  \\
                     0_{\frac{n}{2}-2}
                \end{bmatrix}
            \end{bmatrix}$\\=$\begin{bmatrix}
                \frac{1}{\sqrt{2}}+\frac{1}{(n-2)\sqrt{2}}-\frac{\sqrt{2}}{n-2}  \\
                -\frac{1}{\sqrt{2}}-\frac{1}{(n-2)\sqrt{2}}+\frac{\sqrt{2}}{n-2}  \\
                0_{\frac{n}{2}-2}  \\
                \frac{1}{n-2}-1  \\
               -\frac{1}{n-2}+1  \\
               0_{\frac{n}{2}-2}
            \end{bmatrix}$=$\begin{bmatrix}
                \frac{n-2+1-2}{n-2}  \\
                \frac{-(n-2)-1+2}{n-2}  \\
                0_{\frac{n}{2}-2}  \\
                \frac{1-n+2}{n-2}  \\
               \frac{-1+n-2}{n-2}  \\
                0_{\frac{n}{2}-2}                
            \end{bmatrix}$=$\begin{bmatrix}
                \frac{n-3}{\sqrt{2}(n-2)}  \\
                \frac{3-n}{\sqrt{2}(n-2)}  \\
                0_{\frac{n}{2}-2}  \\
                \frac{3-n}{n-2}  \\
                \frac{n-3}{n-2}  \\
                0_{\frac{n}{2}-2}
            \end{bmatrix}$=$\frac{n-3}{n-2}$$\begin{bmatrix}
                \frac{1}{\sqrt{2}}  \\
               -\frac{1}{\sqrt{2}}  \\
                0_{\frac{n}{2}-2}  \\
                -1  \\
                 1  \\
                 0_{\frac{n}{2}-2}
            \end{bmatrix}$=$\frac{n-3}{n-2}x$.  \\
            As a result, we see that for similar vectors obtained by swapping the location of $-\frac{1}{\sqrt{2}}$ and the concurrent 1, the eigenvalue $\frac{n-3}{n-2}$ has multiplicity $\frac{n-2}{2}$. Now we choose a vector x=$\begin{bmatrix}
                -\sqrt{2}  \\
                 \sqrt{2}  \\
                 0_{\frac{n}{2}-2}  \\
                 -1  \\
                  1  \\
                  0_{\frac{n}{2}-2}  \\
            \end{bmatrix}$, then we see that \\
            $\mathcal{L}(G^c)x=\begin{bmatrix}
            (-\frac{1}{n-2}J_{\frac{n}{2}}+\frac{n-1}{n-2}I_{\frac{n}{2}})\begin{bmatrix}
                 -\sqrt{2}  \\
                  \sqrt{2}  \\
                0_{\frac{n}{2}-2}
            \end{bmatrix}  -\frac{\sqrt{2}}{n-2}(J_{\frac{n}{2}}-I_{\frac{n}{2}})\begin{bmatrix}
                -1  \\
                 1  \\
                 0_{\frac{n}{2}-2}
            \end{bmatrix}  \\
            -\frac{\sqrt{2}}{n-2}(J_{\frac{n}{2}}-I_{\frac{n}{2}})\begin{bmatrix}
                 -\sqrt{2}  \\
                  \sqrt{2}  \\
                  0_{\frac{n}{2}-2}  \\
            \end{bmatrix}+I_{\frac{n}{2}}\begin{bmatrix}
                -1  \\
                 1  \\
                 0_{\frac{n}{2}-2}
            \end{bmatrix}
            \end{bmatrix}$=$\begin{bmatrix}
                \begin{bmatrix}
                    -\sqrt{2}-\frac{\sqrt{2}}{n-2}  \\
                     \sqrt{2}+\frac{\sqrt{2}}{(n-2)}  \\
                   0_{\frac{n}{2}-2}
                \end{bmatrix}+\begin{bmatrix}
                    -\frac{\sqrt{2}}{n-2}  \\
                     \frac{\sqrt{2}}{n-2}  \\
                     0_{\frac{n}{2}-2}  \\
                \end{bmatrix}  \\
                \begin{bmatrix}
                   -\frac{2}{n-2}  \\
                    \frac{2}{n-2}  \\
                    0_{\frac{n}{2}-2}  \\
                \end{bmatrix}+\begin{bmatrix}
                    -1  \\
                     1  \\
                     0_{\frac{n}{2}-2}
                \end{bmatrix}
            \end{bmatrix}$=\\$\begin{bmatrix}
                -\sqrt{2}-\frac{\sqrt{2}}{n-2}-\frac{\sqrt{2}}{n-2}  \\
                 \sqrt{2}+\frac{\sqrt{2}}{n-2}+\frac{\sqrt{2}}{n-2}  \\
                 0_{\frac{n}{2}-2}  \\
                 -\frac{2}{n-2}-1  \\
                  \frac{2}{n-2}+1  \\
                  0_{\frac{n}{2}-2}  
            \end{bmatrix}$=$\begin{bmatrix}
                \frac{-(n-2)\sqrt{2}-2\sqrt{2}}{n-2}  \\
                \frac{(n-2)\sqrt{2}+2\sqrt{2}}{n-2}  \\
                0_{\frac{n}{2}-2}  \\
                \frac{-2-(n-2)}{n-2}  \\
                \frac{2+n-2}{n-2}  \\
                0_{\frac{n}{2}-2}  
            \end{bmatrix}$=$\begin{bmatrix}
                \frac{-n\sqrt{2}}{n-2}  \\
                \frac{n\sqrt{2}}{n-2}  \\
                0_{\frac{n}{2}-2}  \\
                -\frac{n}{n-2}  \\
                 \frac{n}{n-2}  \\
                 0_{\frac{n}{2}-2}                
            \end{bmatrix}$=$\frac{n}{n-2}\begin{bmatrix}
                -\sqrt{2}  \\
                 \sqrt{2}  \\
                 0_{\frac{n}{2}-2}  \\
                 -1  \\
                  1  \\
                  0_{\frac{n}{2}-2}
            \end{bmatrix}$=$\frac{n}{n-2}x$.  \\
            For similar vectors obtained by swapping the position of $\sqrt{2}$ and the concurrent 1, we see that $\frac{n}{n-2}$ is present in the spectrum with a multiplicity $\frac{n-2}{2}$. Now we see that there is an eigenvalue with multiplicity 1 that is still unknown, which we will compute using the trace argument. We let the unknown eigenvalue be $\lambda_j$ and we know that $tr(G^c)=n$. Now, $tr(G^c)$=$\sum^n_{1}\lambda_i$=$\frac{n(n-2)}{2(n-2)}+\frac{(n-3)(n-2)}{2(n-2)}+\lambda_j$. Then we have 
            $n=\frac{n(n-2)}{2(n-2)}+\frac{(n-3)(n-2)}{2(n-2)}+\lambda_j$. 
            Then, $n=\frac{n}{2}+\frac{n-3}{2}+\lambda_j$=$\frac{2n-3}{2}+\lambda_j$. This gives $\lambda_j=n-\frac{2n-3}{2}$=$\frac{3}{2}$. As a result, we see that $\frac{3}{2}$ is present in the spectrum of $G^c$ with multiplicity 1. 
\end{proof}





\bibliographystyle{plain}
\bibliography{sample}

\end{document}